\documentclass[12pt,a4paper]{article}

\usepackage{algorithm}
\usepackage{algpseudocode}
\usepackage{amsfonts}
\usepackage{amsmath}
\usepackage{amssymb}
\usepackage{bbm}
\usepackage[symbol]{footmisc}
\usepackage[margin=1in]{geometry}
\usepackage{mathtools}
\usepackage[amsmath,amsthm,hyperref,thmmarks,thref]{ntheorem}
\usepackage{graphicx}
\usepackage{caption}
\usepackage{float}
\usepackage{hyperref}

\theoremstyle{plain}
\newtheorem{theorem}{Theorem}[section]
\newtheorem{corollary}[theorem]{Corollary}
\newtheorem{lemma}[theorem]{Lemma}
\newtheorem{proposition}[theorem]{Proposition}
\theoremstyle{definition}
\newtheorem{definition}[theorem]{Definition}
\newtheorem{exmp}[theorem]{Example}
\newtheorem{remark}[theorem]{Remark}

\usepackage[font=footnotesize]{caption}

\renewenvironment{proof}[1][Proof. ]{%
    \noindent\textit{#1}%
}{}

\renewcommand{\qed}{%
    \hfill$\blacksquare$%
}

\def\P{\mathbbm{P}}
\def\E{\mathbbm{E}}
\def\V{\text{Var}}

\def\L{\mathcal{L}}
\def\N{\mathbbm{N}}
\def\Z{\mathbbm{Z}}
\def\R{\mathbbm{R}}
\def\1{\mathbbm{1}}
\def\L{\mathcal{L}}
\def\J{\boldsymbol{\mathfrak{J}}}
\def\K{\boldsymbol{\mathfrak{K}}}

\def\cbp{\{\tilde{Z}_n\}}
\def\psdbp{\{Z_n\}}

\def\zt{\tilde{Z}}
\def\xt{\tilde{\xi}}
\def\pt{\tilde{\phi}}
\def\mt{\tilde{m}}
\def\st{\tilde{\sigma}}

\def\poi{\text{Poi}}
\def\bin{\text{Bin}}
\def\ber{\text{Ber}}
\def\geom{\text{Geom}}
\def\nb{\text{NB}}
\def\zig{\text{ZIG}}
\def\zip{\text{ZIP}}

\begin{document}

\title{Linking Population-Size-Dependent and Controlled Branching Processes}
\author{
  Peter Braunsteins\footnote{School of Mathematics and Statistics, University of New South Wales, email: \texttt{p.braunsteins@unsw.edu.au}, ORCID: 0000-0003-1864-0703}, Sophie Hautphenne\footnote{School of Mathematics and Statistics, University of Melbourne, email: \texttt{sophiemh@unimelb.edu.au}, ORCID: 0000-0002-8361-1901}, and James Kerlidis\footnote{School of Mathematics and Statistics, University of Melbourne, email: \texttt{jkerlidis@student.unimelb.edu.au}, ORCID: 0009-0004-4904-0262}
}
\date{}

\maketitle

\begin{abstract}
    Population-size dependent branching processes (PSDBP) and controlled branching processes (CBP) are two classes of branching processes widely used to model biological populations that exhibit logistic growth.
    In this paper we develop connections between the two, with the ultimate goal of determining when a population is more appropriately modelled with a PSDBP or a CBP.
    In particular, we state conditions for the existence of equivalent PSDBPs and CBPs, we then consider the subclass of CBPs with deterministic control functions (DCBPs), stating a necessary and sufficient condition for DCBP-PSDBP equivalence.
    Finally, we derive an upper bound on the total variation distance between non-equivalent DCBPs and PSDBPs
	with matching first and second moments and equal initial population size,
	and show that under certain conditions this bound tends to zero as the initial population size becomes large.
\end{abstract}

\section{Introduction}\label{sec:introduction}

Branching processes are popular models for populations where individuals reproduce and die according to probabilistic rules. They have been applied successfully to address real-world problems in various domains, particularly in biology and conservation ecology, see for example \cite{haccou05,jagers75}. Branching processes are flexible models: they can be discrete- or continuous-time, single- or multitype, discrete- or continuous-state, among others. Here we focus on discrete-time single-type branching processes.

The simplest branching process is the \textit{Galton--Watson branching process} (GWBP).
A GWBP $\{X_n\}_{n\in\N_0}$, with initial state $X_0 = x \in \N_1$, evolves at each time-step (or \textit{generation}) $n$ according to the recursive equation
\begin{equation}\label{eqn:gwbp}
	X_n = \sum_{i=1}^{X_{n-1}} \xi_{n,i}, \quad n \geq 1,
\end{equation}
where the $\{\xi_{n,i}\}_{n,i\in\N_1}$ are i.i.d.\
non-negative integer-valued random variables that share a common distribution, known as the \textit{offspring distribution}. While GWBPs exhibit only exponential growth, many biological populations instead grow logistically. This makes makes the GWBP insufficiently flexible for modeling such populations.
The \textit{population-size-dependent branching process} (PSDBP) and the \textit{controlled branching process} (CBP) are  extensions of the GWBP
which provide greater flexibility to model a wider range of biological populations.

PSDBPs generalise GWBPs by allowing the offspring distribution to change as the population size does. A PSDBP $\{Z_n\}_{n\in\N_0}$, started from $Z_0 = z_0 \in \N_1$, evolves as
\begin{equation}\label{eqn:psdbp}
	Z_n = \sum_{i=1}^{Z_{n-1}} \xi_{n,i}(Z_{n-1}), \quad n \geq 1,
\end{equation}
where, for each $x\in\N_1$, the $\{\xi_{n,i}(x)\}_{n,i\in\N_1}$ are i.i.d.\ random variables
sharing a distribution with the random variable $\xi(x)$.
The dependence of the offspring distribution on the population size is what provides a PSDBP the flexibility to model macroscopic, population-wide phenomena that a GWBP cannot.
One particularly well-studied example is the case of resource scarcity,
where a PSDBP can be used to model a population that stabilises around a carrying capacity
\cite{braunsteins22, hamza16, hognas19, klebaner93}.
A variant of this model has seen success in modelling DNA replication via the polymerase chain reaction \cite{jagers03}.

PSDBPs offer flexibility in modelling \textit{demographic stochasticity} --- the randomness inherent in individuals in a population giving birth --- through the offspring distribution $\xi(x)$.
However, they lack a built-in mechanism to model random external events, commonly referred to as \textit{environmental stochasticity}. 
In contrast, in additional to demographic stochasticity, CBPs can effectively model external conditions such as random environments, migration, and other factors by introducing a (possibly random) \textit{control function}
that determines the number of individuals in a generation which produce offspring.
A CBP $\{\tilde{Z}_n\}_{n\in\N_0}$, started from $\tilde{Z}_0 = z_0 \in \N_1$, is characterised by the recurrence relation
\begin{equation}\label{eqn:cbp}
	\tilde{Z}_n = \sum_{i=1}^{\tilde{\phi}_n(\tilde{Z}_{n-1})} \tilde{\xi}_{n,i}, \quad n \geq 1,
\end{equation}
where, just as for a GWBP, the $\{\tilde{\xi}_{n,i}\}_{n,i\in\N_1}$ are i.i.d.\
non-negative integer-valued random variables with the same distribution as $\tilde{\xi}$, 
but, unlike in a GWBP, a control function $\tilde{\phi}_n$ moderates the number of parents in generation $n-1$ who reproduce.
An important sub-class of CBPs, and indeed the original formulation for a `controlled branching process' \cite{sebastyanov74},
is the family of CBPs with a \textit{deterministic} control function, which we will refer to as DCBPs.
The presence of a control function has allowed CBPs (including, of course, DCBPs) to be studied in the context
of populations which on average experience immigration in addition to their intrinsic growth \cite{foster71, gonzalez04, zubkov72}. 

Our objective here is to make progress towards answering the following questions: when is it appropriate to model a population using a PSDBP? When should a CBP be used instead?
Our approach will be to consider the dual of these two questions: by determining when a PSDBP and CBP have an equivalent, or approximately equivalent, representation, we will demonstrate cases in which a population modelled by one process can equally-appropriately be modelled by the other --- cases in which a modeller would be indifferent to the use of either a PSDBP or a CBP.

After introducing some basic results in Section \ref{sec:preliminaries}, we will consider the question of exact PSDBP-CBP equivalence in Section \ref{sec:EquivalentBranchingProcesses},
and we will discuss when non-equivalent PSDBPs and CBPs can be considered approximately equivalent in Section \ref{sec:approxequal}. 

While others have previously considered --- at least informally --- the ability of PSDBPs and CBPs to model similar populations \cite{kuster85}, \cite[p.~996]{gonzalez03},
to date no research has been published comparing these two classes of processes.
This work will begin to fill this gap in the literature, and will not only formalise the notion that PSDBPs and CBPs  \textit{can} be used to model similar populations,
but will demonstrate \textit{when} this is the case.

\section{Preliminaries}\label{sec:preliminaries}

The following notation is used throughout this paper: $\N_0 := \{z\in\Z : z \geq 0 \}$, $\N_1 := \{z\in\Z : z > 0 \}$, and $[n] := \{1, \dots, n\}$, for $n \in \N_1$.
We use `gcd' to mean the greatest common divisor, with the convention that $\text{gcd}\{0, z\} = z$. 

\subsection{Population-Size-Dependent Branching Processes}\label{sec:PSDBPs}

Population-size-dependent branching processes (PSDBPs) are discrete-time stochastic processes characterised by the recursive equation
\begin{equation*}
	Z_0 = z_0 \in \N_1, \quad Z_n = \sum_{i=1}^{Z_{n-1}} \xi_{n,i}(Z_{n-1}), \quad n \geq 1. \tag{\ref{eqn:psdbp}}
\end{equation*}
PSDBPs differ from GWBPs in one way:
rather than there being a universal offspring distribution, PSDBPs have an offspring distribution $\xi(z)$ that varies with the population size  $z\in\N_0$.
By considering \eqref{eqn:psdbp}, we see that a PSDBP is a time-homogeneous Markov chain with an absorbing state at zero, much like a GWBP.

The \textit{offspring mean} and \textit{offspring variance} of a PSDBP are denoted respectively by $m(z) := \E\xi(z)$ and $\sigma^2(z) := \V(\xi(z))$, and are assumed to be finite for all $z\in\N_0$.
The conditional moments of a PSDBP can be expressed in terms of its offspring mean and variance, given by
\begin{equation}\label{eqn:psdbpmean}
    \E(Z_n | Z_{n-1} = z) = z \, m(z)
\end{equation}
and
\begin{equation}\label{eqn:psdbpvariance}
    \V(Z_n | Z_{n-1} = z) = z \, \sigma^2(z)
\end{equation}
for $n\in\N_1$. \\

An important subclass of PSDBPs are those with a \textit{carrying capacity}, below which the process tends to grow, and above which the process tends to shrink.
More specifically, we say that a branching process $\psdbp$ has a carrying capacity $K>0$ if $\E(Z_n| Z_{n-1} = z) > z$ for $z<K$, and $\E(Z_n| Z_{n-1} = z) < z$ for $z>K$.
PSDBPs and CBPs are sufficiently flexible to accommodate carrying capacity models, while GWBPs are not.

Branching processes with a carrying capacity are stochastic counterparts to logistic growth models,
and can be used to model populations growing under the effect of resource scarcity.
Two popular models with a carrying capacity are listed below, both of which have deterministic counterparts:

\begin{enumerate}
	\item[(i)] the Beverton-Holt (BH) model:
	\[
		\E(Z_n | Z_{n-1} = z) = \frac{2Kz}{K + z}, \quad K > 0;
	\]
	\item[(ii)] the Ricker model
	\[
		\E(Z_n | Z_{n-1} = z) = z \cdot r^{1 - z/K}, \quad r > 1,\; K > 0.
	\]
\end{enumerate}
Being defined only through their means, neither of these two models emit a unique specification, and in fact there exist both PSDBPs and CBPs satisfying each. 

\begin{figure}[h]
	\centering
	\includegraphics[width=13cm]{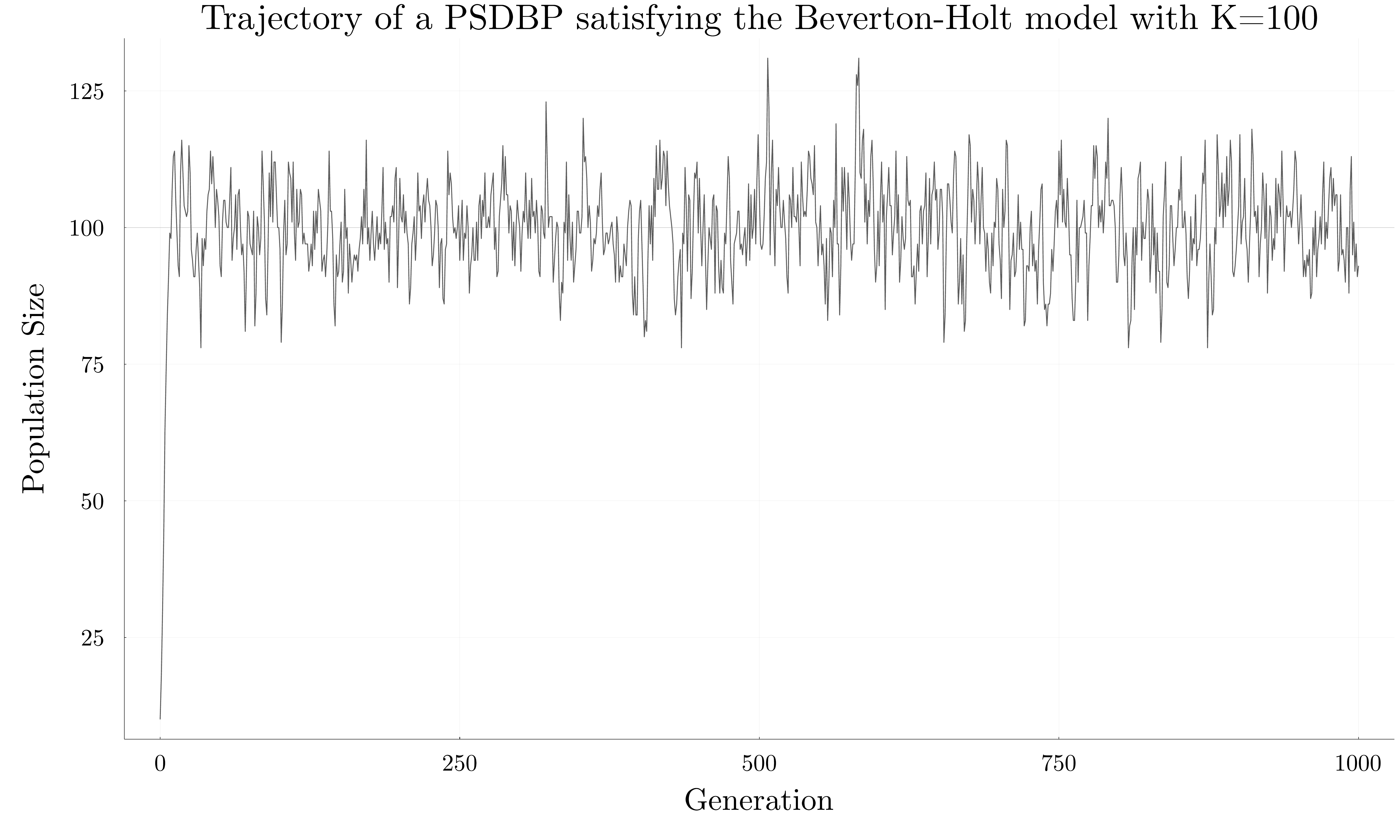}
	\caption{A simulated trajectory, over 1000 generations and started from $Z_0 = 10$, of a PSDBP with $\text{Bin}(2,p(z))$ offspring distribution,
	such that $p(z)$ satisfies the Beverton-Holt model with $K = 100$.}
\label{fig:CarryingCapacity}
\end{figure}
Figure \ref{fig:CarryingCapacity} displays a trajectory from a PSDBP $\psdbp$
with $Z_0 = 10$ and $\xi(z) = \bin\big(2, \frac{100}{100+z}\big)$,
so that $\psdbp$ is a Beverton-Holt model with $K=100$.
We can see that the process rises quickly to reach its carrying capacity,
then lingers around it for the rest of the simulated path.
It has been shown that every PSDBP with a carrying capacity goes extinct a.s.\ \cite[p.~117]{klebaner93},
although the time to extinction for most carrying-capacity models is exponential in the value of $K$ \cite{hognas19}.
Similar behaviour would be observed if we had instead plotted a path from a PSDBP satisfying the Beverton-Holt model with a different offspring distribution,
or from a CBP satisfying the Beverton-Holt model. \\

\subsection{Controlled Branching Processes}\label{sec:CBPs}
Controlled branching processes (CBPs) are also discrete-time stochastic processes that are a modification of a GWBP, this time characterised by the recursive equation
\begin{equation*}
	\zt_0 = z_0 \in \N_1, \quad \zt_n = \sum_{i=1}^{\pt_n(\zt_{n-1})} \xt_{n,i}, \quad n \geq 1. \tag{\ref{eqn:cbp}}
\end{equation*}
For a given generation of size $z\in\N_0$, rather than having each individual reproduce independently to form the next generation,
a function $\pt_n(z)$ of those $z$ individuals reproduce.
This function, known as the \textit{control function}, maps $\N_0$ to the space of random variables supported on $\N_0$,
so that for any $n,z\in\N_0$, $\pt_n(z) \stackrel{d}{=} \pt(z)$ is a random variable.
From \eqref{eqn:cbp} we can see that CBPs are time-homogeneous Markov chains,
but differ from GWBPs and PSDBPs in that they do not necessarily have an absorbing state at zero.
We say that a CBP with $\P(\pt(0) = 0) < 1$ allows \textit{immigration at zero}.

We write the \textit{offspring mean} and \textit{offspring variance} of a CBP $\cbp$ as $\mt := \E\xt$ and $\st := \V(\xi)$ respectively,
from which the conditional moments of the CBP can be expressed as
\begin{equation}\label{eqn:cbpmean}
    \E(\zt_n | \zt_{n-1} = z) = \E\pt(z) \cdot \mt
\end{equation}
and
\begin{equation}\label{eqn:cbpvariance}
    \V(\zt_n | \zt_{n-1} = z) = \E\pt(z) \cdot \st^2 + \V(\pt(z)) \cdot \mt^2
\end{equation}
for $n\in\N_1$. The derivations of these expressions are sketched out in \cite[Proposition~2.2]{gonzalez04}. 

The next remark demonstrates that CBPs are very general processes.
\begin{remark}\label{remark:PSDBPsareCBPs}
    Any time-homogeneous Markov chain supported on $\N_0$ can be represented as a CBP --- by taking $\xt = 1$ a.s. and $\pt$ the distribution of the desired Markov chain.
    Consequently, any PSDBP can be written as a CBP by taking $\xt = 1$ a.s.\ and $\pt(z) = \sum_{i=1}^z \xi_i(z)$.
\end{remark}

In light of \thref{remark:PSDBPsareCBPs}, when modelling populations, we focus on CBPs with specific classes of control functions.
One important class is when $\pt$ has a degenerate distribution. To highlight this setting we denote the control function as $\phi$ rather than $\pt$.
In this case, $\phi$ is a deterministic function mapping the non-negative integers to the non-negative integers.
We call these CBPs with deterministic control functions \textit{deterministically-controlled branching processes} (DCBPs).
When $\cbp$ is a DCBP, the expressions for the offspring mean and variance simplify to
\begin{equation}\label{eqn:dcbpmean}
    \E(\zt_n | \zt_{n-1} = z) = \phi(z) \cdot \mt
\end{equation}
and
\begin{equation}\label{eqn:dcbpvariance}
    \V(\zt_n | \zt_{n-1} = z) = \phi(z) \cdot \st^2.
\end{equation}

Gonz{\'a}lez, del Puerto and Yanev \cite[p.~129]{gonzalez18} also single out three random control functions of importance in the study of CBPs, in terms of a deterministic function $\psi$.
For $z\in\N_0$, these are:
\begin{enumerate}
	\item[(i)] Poisson control function,		\[
			\tilde{\phi}(z) \sim \text{Poi}(\psi(z)), \text{ for } \psi : \N_0 \to \R_{\geq 0},
		\]
	\item[(ii)] Binomial control function,
		\[
			\tilde{\phi}(z) \sim \text{Bin}(\psi(z), q), \text{ for } q \in [0, 1] \text{ and } \psi : \N_0 \to \N_0,
		\]
	\item[(iii)] Negative-binomial control function,
		\[
			\tilde{\phi}(z) \sim \text{NB}(\psi(z), q), \text{ for } q \in [0, 1] \text{ and } \psi : \N_0 \to \R_{\geq 0}.
		\]
\end{enumerate}

\section{Equivalent Branching Processes}\label{sec:EquivalentBranchingProcesses}

We say that two branching processes are \textit{equivalent} if they have the same finite-dimensional distributions (or equivalently here, the same law).
Since the finite-dimensional distributions of a Markov chain can be characterised solely by the chain's initial distribution and transition probabilities \cite[Theorem 5.2.1]{durrett19},
the following result is immediate:

\begin{lemma}\label{BPEquivalence}
	A PSDBP $\psdbp$ and a CBP $\cbp$
	with initial population size $Z_0 = \tilde{Z}_0 = z_0 \in \N_1$ are equivalent if and only if
	\begin{equation*}
		(Z_n | Z_{n-1} = z) \stackrel{d}{=} (\tilde{Z}_n | \tilde{Z}_{n-1} = z)
	\end{equation*}
	for all $n\in\N_1$ and all attainable $z\in\N_0$.
\end{lemma}

For a given initial state $z_0$, we define the set of \textit{attainable} population sizes of a process $\psdbp$ precisely as
\begin{equation}\label{eqn:attainable}
    \mathcal{A}(z_0) := \big\{ x\in\N_0 : \exists\, n \in \N_0 \text{ s.t. } \P(Z_n = x | Z_0 = z_0) > 0 \big\},
\end{equation}
so that if $z \in \mathcal{A}(z_0)$, $z$ has a non-zero probability of being reached. 

\thref{BPEquivalence} justifies our claim in \thref{remark:PSDBPsareCBPs} that all PSDBPs are CBPs.
The converse, however, is not true.
Since PSDBPs necessarily have an absorbing state at a population size of zero, while CBPs do not, we can easily construct a CBP that has no equivalent PSDBP.
One such CBP is depicted in Figure \ref{fig:RandomExtinction}.

\begin{remark}\label{noimmigration}
    No CBP which allows immigration at zero is able to be expressed as a PSDBP.
\end{remark}

\begin{figure}[!htb]
	\centering
	\includegraphics[width=13cm]{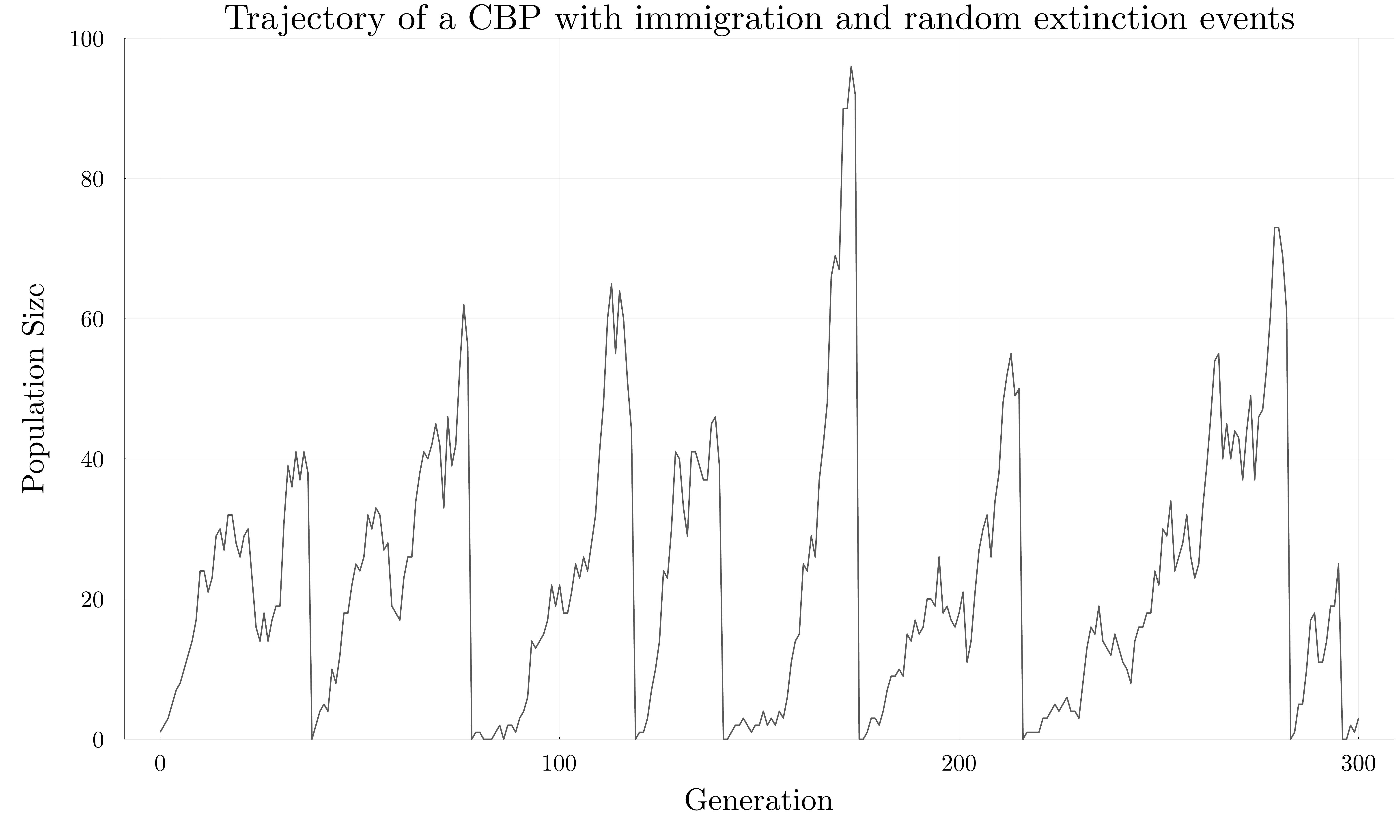}
	\caption{A simulated trajectory, over 300 generations and started from $\tilde{Z}_0 = 1$, of a CBP with $\tilde{\phi}(z) \sim (z+1) \cdot \text{Ber}\big( e^{-z/1000} \big)$
    and $\tilde{\xi} \sim \text{Bin}(5,1/5)$, such that the process experiences immigration at zero.}
 \label{fig:RandomExtinction}
\end{figure}

\subsection{Equivalent PSDBPs and CBPs}\label{sec:equivalentprocesses}

Given a CBP $\cbp$, we consider the problem of determining the existence of an equivalent PSDBP.
Informally speaking and in light of \thref{BPEquivalence}, this is the same as asking whether we can, for all attainable $z$, `divide' the random variable $\sum_{i=1}^{\pt(z)}\xt_i$ into $z$ i.i.d.\ components.

The divisibility of random variables is a well-studied concept \cite{steutel03}.
For $n\in\N_1\setminus\{ 1\}$, a random variable $X$  on $\N_0$ is said to be \textit{n-divisible} if there exist $n$ i.i.d.\ random variables $\{X_i^{(n)}\}_{i\in[n]}$ such that $X \stackrel{d}{=} \sum_{i=1}^n X_i^{(n)}$.
It is said to be \textit{divisible} if there is at least one $n$ such that $X$ is $n$-divisible,
and to be \textit{infinitely divisible} if $X$ is $n$-divisible for all $n$.

It is well-known that random variables with Poisson, geometric, or negative binomial distributions are infinitely divisible \cite[p.~28]{steutel03},
while binomial-distributed random variables are divisible (as the sum of i.i.d.\ Bernoulli random variables), but not infinitely divisible, as no non-degenerate bounded random variable is infinitely divisible~\cite[p.~4]{steutel03}.
The distribution of $\sum_{i=1}^{\pt(z)}\xt_i$, however, often does not have such a convenient form.
Instead we focus on the control function $\pt$ of the CBP and ask whether we can divide $\pt(z)$ into $z$ i.i.d.\ components for all attainable $z\geq 1$.
This question is formulated in terms of the family of random variables $\pt(z)$ indexed by $z$. To capture this concept, we introduce a new type of divisibility.

\begin{definition}\label{def:zdivisiblecontrolfunctions}
	Let $\cbp$ be a CBP with $\zt_0 = z_0$ and control function $\pt$.
	We say that the control function \textit{\textbf{$\boldsymbol{\pt}$ is $\boldsymbol{\zt}$-divisible}}
	with respect to $\cbp$ if $\pt(0) = 0$
	and $\pt(z)$ is $z$-divisible for all attainable $z\in\N_1$.
\end{definition}

In words, a control function is $\zt$-divisible if it does not allow immigration at zero and is always divisible by the current population size.
In particular, if $\pt(0) = 0$ and $\pt(z)$ is infinitely divisible for all attainable $z$, then $\pt$ is $\zt$-divisible.
For example, if for all $z\in\N_0$ we have $\pt(z) \sim \poi(\psi(z))$, $\psi : \N_0 \to \R_{\geq 0}$, then $\pt$ is $\tilde{Z}$-divisible as long as $\psi(0) = 0$.

The next proposition uses this definition to state a sufficient condition for CBP-PSDBP equivalence.

\begin{proposition}\label{divisiblecontrolfunctions}
	Let $\cbp$ be a CBP with a $\tilde{Z}$-divisible control function, $\pt$.
	Then $\cbp$ can be expressed as a PSDBP.
\end{proposition}

Given \thref{BPEquivalence}, and since $\zt$-divisibility guarantees that there exist i.i.d.\ $\{\zeta_i(z)\}_{i\in [z]}$
such that $\pt(z) \stackrel{d}{=} \sum_{i=1}^z \zeta_i(z)$ for all attainable $z$ (i.e.\ for all $z\in\mathcal{A}(z_0)$ if $\tilde{Z}_0 = z_0 \in\N_1$),
\thref{divisiblecontrolfunctions} follows from the rearrangement
\[
    \sum_{i=1}^{\pt(z)} \xt_i \stackrel{d}{=} \sum_{i=1}^z \sum_{j=1}^{\zeta_i(z)} \xt_j
    \stackrel{d}{=} \sum_{i=1}^z \xi_i(z),
\]
for $\xi_i(z) := \sum_{j=1}^{\zeta_i(z)} \xt_j$.

\begin{corollary}\label{PoissonNBBPEquivalence}

    If a CBP $\cbp$ has, for all attainable $z\in\N_0$,
    \[
        \tilde{\phi}(z) \sim \text{Poi}(\psi(z)) \quad\text{or}\quad
        \tilde{\phi}(z) \sim \text{NB}(\psi(z), q),
    \]
    where $\psi : \N_0 \to \R_{\geq 0}$ with $\psi(0) = 0$, and $q\in [0, 1]$,
	then $\cbp$ be expressed equivalently as a PDSBP.
\end{corollary}

This result follows directly from \thref{divisiblecontrolfunctions},
since both the Poisson and negative-binomial distributions are infinitely divisible.
In fact, since \thref{noimmigration} precludes any equivalence in the case of $\psi(0) \neq 0$,
\thref{PoissonNBBPEquivalence} completely characterises the conditions for PSDBP-CBP equivalence for two of the three important families of random control functions listed in Section \ref{sec:CBPs}.

The third important class of random control functions, binomial control functions,
are not infinitely divisible.
\thref{divisiblecontrolfunctions} still allows us to state a sufficient condition in this case:

\begin{corollary}\label{BinomialDivisibleEquivalence}
    Let $\cbp$ be a CBP with control function
    \[
        \pt(z) \sim \text{Bin}(\psi(z), q)
    \]
    for $z\in\N_0$, where $\psi : \N_0 \to \N_0$ and $q \in [0, 1]$.
    Then there exists an equivalent PSDBP if $\psi(0) = 0$ and $\psi(z)/z \in \N_0$ for all attainable $z\in\N_1$.
\end{corollary}

Because \thref{BinomialDivisibleEquivalence} does not give an `if and only if' condition, it does not provide a complete characterisation of  equivalence for a binomial control function.
To say more, we need to look at the CBP as a whole, rather than just at its control function.

Consider the class of CBPs with binomial control functions that are \textit{not}
$\zt$-divisible.
We might expect that such a CBP could still be expressed as an equivalent PSDBP if it has an infinitely divisible offspring distribution. 
However, this turns out to not necessarily be the case.

\begin{proposition}\label{PoissonOffspringNoEquivalence}
	Consider the CBP $\cbp$ with $\zt_0 = z_0 \in \N_1$,
	\[
		\pt(z) \sim \bin\big(\psi(z), p(z)\big) \quad\text{and}\quad \xt \sim \poi(\lambda)
	\]
	for $z\in\N_0$, where $\psi: \N_0 \to \N_0$,  $p: \N_0 \to (0,1)$, and $\lambda > 0$.
	If $\pt(z)$ is not $\zt$-divisible, and $\psi(z) \geq 1$ for all $z \geq 1$,
	then $\cbp$ cannot be expressed equivalently as a PSDBP.
\end{proposition}

We can compare this result with \cite[Proposition VI.6.2]{steutel03}, in which Steutel and van Harn show that a mixture of different Poisson distributions is not infinitely divisible if the mixing distribution has a non-degenerate, bounded support.

On the other hand, we can also produce an example of a CBP without a $\zt$-divisible control function that \textit{can} be equivalently expressed as a PSDBP.

\begin{proposition}\label{GeometricOffspringEquivalence}
	Consider the CBP $\cbp$ with $\zt_0 = z_0 \in \N_1$, 
	\[
		\pt(z) \sim \bin(\psi(z), p(z)) \quad\text{and}\quad \xt \sim \geom(q)
	\]
	for $z\in\N_0$, where $\psi: \N_0 \to \N_0$, $p: \N_0 \to (0,1)$, and $q \in (0,1)$.
	Then as long as $\psi(0) = 0$, $\cbp$ can be expressed equivalently as a PSDBP.
\end{proposition}

It is worthwhile to note that although we have proven that the CBP in \thref{GeometricOffspringEquivalence}
\textit{can} be expressed equivalently as a PSDBP,
the form of this equivalent PSDBP may not be very neat: its offspring distribution will be formed from the components
of a divided zero-inflated geometric distribution, which do not have a known closed form. 

Having demonstrated cases where PSDBP-CBP equivalence is impossible, and cases where it is possible,
we leave open the question of classifying the class of CBPs without a $\zt$-divisible control function that are amenable to equivalence.

\subsection{Equivalent PSDBPs and DCBPs}\label{sec:deterministiccontrolfunctions}

In Section \ref{sec:CBPs} we introduced the family of DCBPs as the set of CBPs with deterministic control functions.
The same results developed in Section \ref{sec:equivalentprocesses} apply when the CBP in question is a DCBP;
we can simply take the control function to have a degenerate distribution when determining $\tilde{Z}$-divisibility.
For example, the DCBP $\cbp$ with $\phi(z) = 2z$ and $\xt \sim \ber(1/2)$
has a $\zt$-divisible control function,
so by \thref{divisiblecontrolfunctions} it can be expressed as a PSDBP.
Indeed this is the case, since $(\zt_n | \zt_{n-1} = z) \stackrel{d}{=} \sum_{i=1}^z \xi_{n,i}$, where $\xi \sim \bin( 2, 1/2 )$. 

Unlike for CBPs, the set of DCBPs is not a superset of the PSDBPs.
Similarly, the set of PSDBPs is not a superset of the DCBPs.
Indeed, for example, the PSDBP $\psdbp$ with $\xi(z) \sim \ber\big( z/(z+1) \big)$, $z\in\N_0$, cannot be expressed as a DCBP,
while the DCBP $\cbp$ with $\phi(z) = \max\{z-1, 0\}$, $z\in\N_0$, and $\xt \sim \ber( 1/2)$ cannot be expressed as a PSDBP. 

The simplified structure of DCBPs allows us to strengthen \thref{divisiblecontrolfunctions} into a necessary and sufficient condition, based on the following definition:

\begin{definition}\label{OffspringDistributionZDivisibility}\label{offspringdistributionzdivisibility}
	Let $\cbp$ be a DCBP with control function $\phi$,
	offspring distribution $\xt$, and initial population size $z_0$.
	We say that \textit{\textbf{the process $\boldsymbol{\cbp}$ is $\boldsymbol{\zt}$-divisible}}
	if $\phi(0) = 0$ and, for all values
	\[
		y \in \mathcal{Y}_{\tilde{Z}} := \left\{
			\frac{z}{\text{gcd}\{\phi(z), z\}} \,:\, z \in \mathcal{A}(z_0) \setminus \{0\}
		\right\},
	\]
	$\xt$ is $y$-divisible.
\end{definition}

In simpler terms, \thref{offspringdistributionzdivisibility} says that for a DCBP to be $\zt$-divisible, 
$\xt$ must be divisible by all the prime factors of $z$ that $\phi(z)$ is not divisible by, for any attainable $z$ (except zero).

\begin{remark}
    When $\cbp$ is a DCBP, \thref{offspringdistributionzdivisibility} extends \thref{def:zdivisiblecontrolfunctions}: any DCBP with a $\zt$-divisible control function is a $\zt$-divisible DCBP.
\end{remark}

\begin{proposition}\label{DCBPEquivalenceConditions}
	Let $\cbp$ be a DCBP with control function $\phi(\cdot)$ and offspring distribution $\xt$.
	Then $\cbp$ can be expressed as a PSDBP if and only if $\cbp$ is $\zt$-divisible.
\end{proposition}

We demonstrate the application of \thref{DCBPEquivalenceConditions} in the following example:

\begin{exmp}\;
\begin{itemize}
    \item[(i)] Consider the DCBP $\cbp$ with $\xt \sim \bin\big(2, \frac{1}{2}\big)$ and
        \[
            \phi(z) = \begin{cases}
    			z, & z \text{ odd} \\ 
    			\frac{z}{2}, & z \text{ even}.
    		\end{cases}
        \]
        Here, $\frac{z}{\text{gcd}[\phi(z), z]} = 1$ when $z$ is odd, and $\frac{z}{\text{gcd}[\phi(z), z]} = 2$ when $z$ is even,
        so $\mathcal{Y}_{\tilde{Z}} = \{1, 2\}$.
        Since $\xt$ is 2-divisible, \thref{DCBPEquivalenceConditions} guarantees that there is a PSDBP equivalent to $\cbp$.
        Indeed, the PSDBP $\psdbp$ with
		\[
			\xi(z) \sim \begin{cases}
				\text{Bin}\big(2, \frac{1}{2}\big), & z \text{ odd} \\ 
				\text{Ber}\big(\frac{1}{2}\big), & z \text{ even}
			\end{cases}
		\]
		is exactly equivalent to $\cbp$.
    \null\hfill$\blacklozenge$

    \item[(ii)] Consider now the DCBP $\cbp$ with $\phi(z) = \max\{z-1,0\}$, $z\in\N_0$
        and $\xt \sim \poi(\lambda)$, $\lambda > 0$. Here $\mathcal{Y}_{\tilde{Z}} = \N_1$,
        but since $\xt$ is infinitely divisible, we see that $\cbp$ is $\zt$-divisible,
        and hence \thref{DCBPEquivalenceConditions} tells us that $\cbp$ can be equivalently expressed as a PSDBP --- we can then find that this PSDBP is such that $\xi(z) \sim \poi( (z-1)\lambda / z )$.
    \null\hfill$\blacklozenge$
\end{itemize}
\end{exmp}

\section{Approximately Equivalent Branching Processes}\label{sec:approxequal}

While we saw in Section \ref{sec:EquivalentBranchingProcesses} that we can often find pairs of equivalent PSDBPs and CBPs, we also observed cases where exact equivalence is not possible.
In practical applications, however, the lack of exact equivalence becomes less significant when the two processes are close in distribution.

How might we then find PSDBPs and CBPs that are close in distribution?
The requirement we impose is that both $(Z_n | Z_{n-1} = z)$ and $(\tilde{Z}_n | \tilde{Z}_{n-1} = z)$ have the same mean and variance for all $z$.
After all, the population size of a given generation is generated, for both CBPs and PSDBPs, as a sum of i.i.d.\ random variables; as a result, if we require that both sums have the same mean and variance, we may expect a central-limit-theorem-like result to apply as the number of summed terms gets large.
Informally, we expect that the paths of non-equivalent PSDBPs and CBPs with the same conditional mean and variance become more and more alike as their initial population size increases.

\begin{definition}
    We say that a PSDBP $\psdbp$ and a CBP $\cbp$, both with initial population size $z_0$, 
    \textit{\textbf{match}}, or have \textit{\textbf{matching moments}}, if for $n \in \N_1$ they satisfy
    \begin{align}
    	\E(Z_n | Z_{n-1} = z) &= \E(\tilde{Z}_n | \tilde{Z}_{n-1} = z), \label{eqn:matchingmean} \\
    	\V(Z_n | Z_{n-1} = z) &= \V(\tilde{Z}_n | \tilde{Z}_{n-1} = z), \label{eqn:matchingvariance}
    \end{align}
    for all attainable $z\in\N_0$.
\end{definition}

We formalise the notion of closeness in distribution in terms of convergence in the total variation distance. We show that, under certain conditions, the total variation distance between a PSDBP and a CBP with matching moments approaches zero as the initial population size of the processes, $z$, approaches infinity. Moreover, we show that this convergence occurs at a rate of $1/\sqrt{z}$.

\begin{definition}
	For random variables $X$ and $Y$ taking values in a space $\mathcal{X}$, the \textit{\textbf{total variation distance (TVD)}}
	between their distributions is 
	\[
		|| \L_X - \L_Y ||_{TV} := \sup_{A \subseteq \mathcal{X}} | \L_X(A) - \L_Y(A) |
		\equiv \sup_{A \subseteq \mathcal{X}} | \P(X \in A) - \P(Y \in A) |,
	\]
    where $A$ is an event, and $\L_X$ denotes the law of the random variable $X$.
\end{definition}

\begin{remark}
    If $\mathcal{X}$ is a countable space (e.g.\ $\mathcal{X} = \N_0^k$ for some $k \geq 1$), the TVD has the following equivalent representations:
    \begin{align}
        || \L_X - \L_Y ||_{TV} &= \frac{1}{2} \sum_{n\in\mathcal{X}} | \P(X = n) - \P(Y = n) |, \label{eqn:TVDEquivalentSum} \\
        || \L_X - \L_Y ||_{TV} &= 1 - \sum_{n\in\mathcal{X}} \P(X = n) \wedge \P(Y = n); \label{eqn:TVDOneMinusEquivalentSum}
    \end{align}
    see for example \cite[Chapter 4]{levin17}.
\end{remark}

\subsection{Distance between PSDBPs and DCBPs with Matching Moments}\label{sec:distancebetweenmatchingPSDBPsandDCBPs}

The simplest situation in which to bound the TVD between a PSDBP and a CBP with matching moments is when the CBP in question has a deterministic control function.
Given a PSDBP-DCBP pair with matching moments, for which necessary and sufficient conditions are stated in Section \ref{sec:matchingmomentsPSDBPandDCBP},
we impose the following three regularity conditions, with $\tilde{\rho} := \E|\tilde{\xi} - \tilde{m}|^3$ and $\rho(x) := \E|\xi(x) - m(x)|^3$. \\

\noindent\textbf{TVD Bound Regularity Conditions:}
\begin{enumerate}
    \item[\;\textbf{(C1)}] There exists $h \in \R_{>0}$ such that $\phi(x) \geq h \cdot x$ for all $x\in\N_1$.
    \item[\textbf{(C2)}] There exists $R \in \R_{>0}$ such that $\tilde{\rho} \vee \rho(x) \leq R$ for all $x \in \N_1$.
    \item[\textbf{(C3)}] There exists $\eta \in \R_{>0}$ such that every $x\in \N_1$ has a corresponding $n_x \in \N_1$ with $\P(\xi(x) = n_x) \wedge \P(\xi(x) = n_x - 1) \geq 2\eta$, and there is an $n\in\N_1$ such that $\P(\tilde{\xi} = n) \wedge \P(\tilde{\xi} = n - 1) \geq 2\eta$. \\
\end{enumerate}

Under these conditions, the following lemma provides an upper bound on the TVD between the one-step distributions of a PSDBP and DCBP with matching moments and started from the same population size $z$.
The bound is decreasing in $z$, with optimal rate $O(z^{-1/2})$.

\begin{lemma}\label{OneStepTVDBound}
	Let $\psdbp$ be a PSDBP, and $\cbp$ a DCBP, with matching first and second moments.
    Then if $\psdbp$ and $\cbp$ satisfy conditions (C1), (C2) and (C3), we have
	\[
		|| \L_{Z_1|Z_0 = z} - \L_{\tilde{Z}_1|\tilde{Z}_0 = z} ||_{TV} \leq \J(z)
	\]
	for all $z \in \N_1$, where $\J : \R_{>0} \to \R_{>0}$ is a monotonically decreasing function defined by
	\[
		 \J(z) := \frac{\sqrt{2}\left( 3R + 2(1 + h)\tilde{\sigma}^2 \right)}{\tilde{\sigma}^2 (h\wedge 1) \sqrt{\pi \eta z}}
			+ \frac{\left( 5\sqrt{2\pi} + \frac{3\pi}{2} \right) (1 + h) R + h\tilde{\sigma}^2}{\tilde{\sigma}^3\sqrt{2\pi h^3 z}}.
	\]
\end{lemma}

\thref{OneStepTVDBound} relies on a comparison of $\psdbp$ and $\cbp$ to a discretised normal distribution that leverages the results of \cite{chen11}. 
Bounds of the same order can also be found by replacing the discretised normal with another distribution,
for example the translated Poisson distribution, as in \cite{barbour99},
or the distribution proposed in \cite{goldstein06}.

Using an inductive argument, we can extend this bound from the one-step case to consider longer trajectories of $\psdbp$ and $\cbp$.

\begin{proposition}\label{KStepTVDBound}
	Let $\psdbp$ be a PSDBP, and $\cbp$ a DCBP, with matching first and second moments.
    Then if $\psdbp$ and $\cbp$ satisfy conditions (C1), (C2) and (C3), we have, for any $z \in \N_1$, $k\in\N_1$, and $\alpha \in (0,1)$,
	\[
		||\L_{(Z_1, \dots, Z_k)|Z_0 = z} - \L_{(\tilde{Z}_1, \dots, \tilde{Z}_k)|\tilde{Z}_0 = z}||_{TV}
			\leq \sum_{i=0}^{k-1} \J\big( ( \alpha \tilde{m} h )^i \cdot z \big)
			+ \frac{\tilde{\sigma}^2}{(1 - \alpha)^2 \tilde{m}^2 h \cdot z} \cdot \sum_{i=0}^{k-2} ( \alpha \tilde{m} h )^{-i}.
	\]
\end{proposition}

\begin{remark}
   For any fixed $k\in\N_1$, and regardless of the value of $\alpha$, \thref{KStepTVDBound} provides a bound on the TVD between $\psdbp$ and $\cbp$ that decreases in $z$, with optimal rate
    \[
    	||\L_{(Z_1, \dots, Z_k)|Z_0 = z} - \L_{(\tilde{Z}_1, \dots, \tilde{Z}_k)|\tilde{Z} = z}||_{TV} = O\bigg( \frac{1}{\sqrt{z}} \bigg),
    \]
    a result inherited from \thref{OneStepTVDBound}, since $\J(z) = O(z^{-1/2})$.
    This implies $$||\L_{(Z_1, \dots, Z_k)|Z_0 = z} - \L_{(\tilde{Z}_1, \dots, \tilde{Z}_k)|\tilde{Z} = z}||_{TV} \to 0\quad\text{as $z\to\infty$,} $$
    so in the population limit the two processes will be indistinguishable over a fixed number of $k$ generations.
\end{remark}

\begin{remark}\label{MatchingZIPandNB}
    Although stated in terms of DCBPs, \thref{OneStepTVDBound} and \thref{KStepTVDBound} apply to certain CBPs with random control functions as well.
    For example, suppose that  $\lambda > 0$, $M \in \N_1$, and consider the CBP $\cbp$ with
    \[
        \pt(z) \sim \bin\left((z + M)\1_{\{z>0\}},\; \frac{1}{\lambda}\right) \quad \text{and} \quad
    	\xt \sim \poi\left( \lambda \right),
    \]
    and the PSDBP $\psdbp$ with
    \[
    	\xi(z) \sim \nb\left(\frac{z+M}{z},\; \frac{1}{\lambda} \right).
    \]
    \thref{PoissonOffspringNoEquivalence} tells us that there is no PSDBP equivalent to $\cbp$,
    but we can verify that both process have matching moments.
    Since, for any $z\in\N_1$,
    \[
        \sum_{i=1}^{\pt(z)} \xt_i \stackrel{d}{=} \sum_{i=1}^{\nu(z)} \zeta_i,
    \]
    where $\nu(z) := (z + M)\1_{\{z>0\}}$ and $\zeta_i \stackrel{i.i.d.}{\sim} \zip\left( 1 - \frac{1}{\lambda},\, \lambda \right)$
    ($\zip$ here refers to the zero-inflated Poisson distribution),
    we can express $\cbp$ as a DCBP ---
    allowing us to use our framework in \thref{KStepTVDBound} to determine a bound on the TVD between the two processes.
\end{remark}

\begin{remark}
    It is possible to generalise \thref{KStepTVDBound} in several directions,
    at the expense of more delicate regularity conditions.
    Two such generalisations are to the case of a PSDBP and a CBP with moments that do not match exactly,
    and to the case of a PSDBP and a hybrid DCBP-PSDBP, that is, a process that evolves by the recurrence relation
    \begin{equation*}
        \zt_n = \sum_{i=1}^{\phi(\zt_{n-1})} \xt_{n,i}(Z_{n-1}), \quad n \geq 1,
    \end{equation*}
    where, for every $x \in \N_0$, $\{\xt_{n,i}(x)\}_{n,i\in\N_1}$ is a family of i.i.d.\ $\N_0$-valued random variables
    sharing a common distribution with a random variable $\xi(x)$,
    and $\phi: \N_0 \to \N_0$ is a deterministic control function.
\end{remark}

\begin{corollary}\label{ClosedFormKStepTVDBound}
	For each pair of processes $\psdbp$ and $\cbp$ with matching first and second moments and satisfying conditions (C1), (C2) and (C3),
	there exist constants $b, c_1, c_2 \in \R_{>0}$ such that
	\[
		||\L_{Z_1|Z_0 = z} - \L_{\tilde{Z}_1|\tilde{Z}_0 = z}||_{TV} \leq \frac{b}{\sqrt{z}},
	\]
	and for $k > 1$,
	\[
		||\L_{(Z_1, \dots, Z_k)|Z_0 = z} - \L_{(\tilde{Z}_1, \dots, \tilde{Z}_k)|\tilde{Z}_0 = z}||_{TV}
		\leq \frac{c_1 \big| 1 - (\alpha\tilde{m}h)^{-\frac{k}{2}} \big|}{\sqrt{z}}
			+ \frac{c_2 \big| 1 - (\alpha\tilde{m}h)^{-k+1} \big|}{z}.
	\]
\end{corollary}

\thref{ClosedFormKStepTVDBound} shows that the behaviour of the bound differs depending on the value of $\tilde{m}h$:
if $\tilde{m}h \leq 1$, then the restriction $\alpha \in (0,1)$ will mean that $\alpha\tilde{m}h \leq 1$ regardless of the value of $\alpha$ chosen.
In this case, the bound is exponentially increasing in $k$.
However, if $\tilde{m}h > 1$ (which roughly translates to the two processes being supercritical at every population size),
then we can take $\alpha \in \left( 1/\tilde{m}h, 1 \right)$, so that $\alpha\tilde{m}h > 1$. In that case,
\begin{align*}
	||\L_{\psdbp|Z_0 = z} - \L_{\cbp|\tilde{Z}_0 = z}||_{TV}
	& = \lim_{k\to\infty} ||\L_{(Z_1, \dots, Z_k)|Z_0 = z} - \L_{(\tilde{Z}_1, \dots, \tilde{Z}_k)|\tilde{Z}_0 = z}||_{TV} \\
	&\leq \lim_{k\to\infty} \frac{c_1 \big( 1 - (\alpha\tilde{m}h)^{-\frac{k}{2}} \big)}{\sqrt{z}}
		+ \frac{c_2 \big( 1 - (\alpha\tilde{m}h)^{-k+1} \big)}{z} \\
	&\leq \frac{c_1}{\sqrt{z}} + \frac{c_2}{z} \\
	&=  O\bigg( \frac{1}{\sqrt{z}} \bigg).
\end{align*}
Crucially, this shows that, whenever $\tilde{m}h > 1$, we can choose $\alpha$ so that the bound does not increase as $k$ gets larger;
in the limit as $z$ increases to infinity, the two processes $\psdbp$ and $\cbp$ are indistinguishable
over their entire, infinitely-long trajectories.

\subsection{Conditions to Match Moments of PSDBPs and DCBPs}\label{sec:matchingmomentsPSDBPandDCBP}

We have assumed throughout Section \ref{sec:distancebetweenmatchingPSDBPsandDCBPs} that it is possible to find PSDBPs and DCBPs with matching moments.
This is indeed a weaker condition than equivalence, as displayed by \thref{MatchingZIPandNB},
but as we will see is not a trivial requirement. \\

When the CBP in question is a DCBP, Equations \eqref{eqn:matchingmean} and \eqref{eqn:matchingvariance} simplify to the following:

\begin{proposition}\label{MatchingMomentsEquations}
	For a PSDBP and a DCBP to have matching mean and variance, we require that
	\begin{equation}\label{eqn:DCBPPSDBPMatchingMean}
		z \cdot m(z) = \tilde{m} \cdot \phi(z)
    \end{equation}
    and
    \begin{equation}\label{eqn:DCBPPSDBPMatchingVariance}
        z \cdot \sigma^2(z) = \tilde{\sigma}^2 \cdot \phi(z)
	\end{equation}
	for all $z\in\N_0$.
\end{proposition}

From Equations \eqref{eqn:DCBPPSDBPMatchingMean} and \eqref{eqn:DCBPPSDBPMatchingVariance} we can derive necessary and sufficient conditions to match both a PSDBP to a given DCBP, and a DCBP to a given PSDBP.

\begin{proposition}\label{WhenPSDBPMustMatchDCBP}
	A PSDBP $\psdbp$ can be found to match the mean and variance of a DCBP
	$\cbp$ if and only if $\cbp$ satisfies
	\[
		\frac{\tilde{\sigma}^2 \cdot \phi(z)}{z} \geq d(z) (1-d(z))
	\]
	for all $z \in \N_1$, where $d(z) := \frac{\tilde{m} \cdot \phi(z)}{z} - \left\lfloor \frac{\tilde{m} \cdot \phi(z)}{z} \right\rfloor$.
\end{proposition}

\begin{proposition}\label{WhenDCBPMustMatchPSDBP}
	A DCBP $\cbp$ can be found to match the mean and variance of a PSDBP $\psdbp$
	if and only if either the PSDBP has $\sigma^2(z) = 0$ for all attainable $z \in \N_1$, or if the following requirements on $\psdbp$ hold:
	\begin{enumerate}
		\item[(i)] there exists a constant $k \in \R_{>0}$ such that, for all attainable $z\in\N_1$, $m(z) = k \cdot \sigma^2(z)$,
		\item[(ii)] the set $H := \big\{ h \in \R_{>0} : m(z) \in h \cdot \N_0 \text{ for all attainable } z \in \N_1 \big\}$ is non-empty,
		\item[(iii)] there exists $h' \geq \sup_{h>1} \{ h \in H \}$ such that $\frac{h'}{k} \geq \big( h' - \lfloor h' \rfloor \big) \big( 1 - h' + \lfloor h' \rfloor \big)$.
	\end{enumerate}
\end{proposition}

We can use these propositions to determine whether or not it is possible to match moments to a given process.

\begin{figure}[!htb]
\centering
\includegraphics[width=12cm]{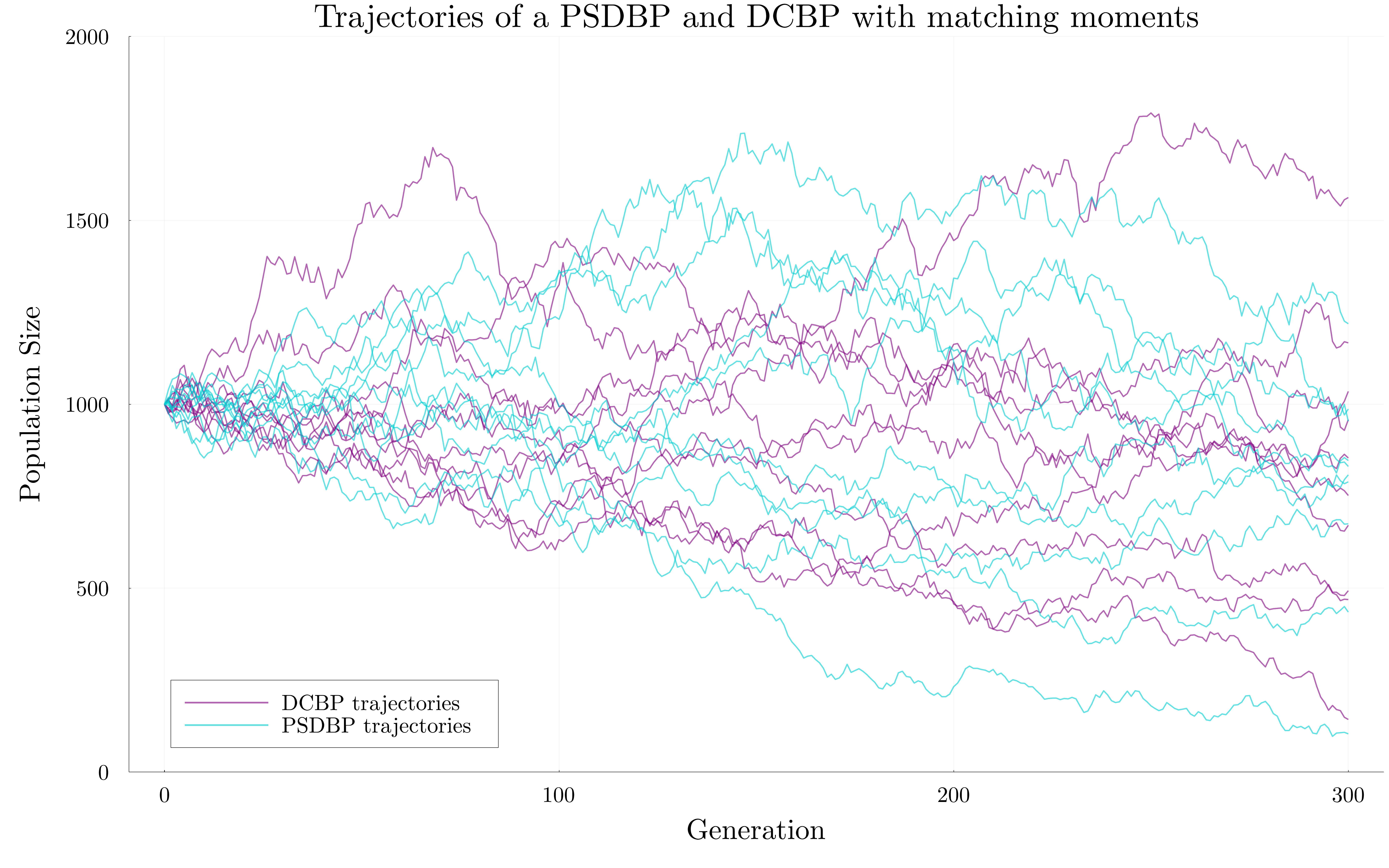}
\caption{Ten simulated trajectories of each a PSDBP and a DCBP with matching moments and $Z_0 = \tilde{Z}_0 = 1000$.
The DCBP has $\phi(z) = \max\{z-1, 0\}$ and $\tilde{\xi} \sim \text{Bin}(2, 1/2)$,
while the PSDBP has $\xi(0) = \xi(1) = 0$, and for $z \geq 2$, $\P(\xi(z) = 0) = \frac{z^2+z+2}{4z^2}$,
$\P(\xi(z) = 1) = \frac{z^2+z-2}{2z^2}$, and $\P(\xi(z) = 2) = \frac{z^2-3z+2}{4z^2}$.}
\label{fig:MatchingMoments}
\end{figure}

\begin{exmp}\label{DCBPWithNoMatchingPSDBP}
	Consider a DCBP $\cbp$ with control function $\phi(z) = \max\{z-1,0\}$, for $z\in\N_0$, and $\xt \sim \ber(1/2)$ or $\xt \sim \bin(2, 1/2)$.
    In both of these offspring distributions, we can use \thref{DCBPEquivalenceConditions} to show that $\cbp$ has no equivalent PSDBP.
    \begin{enumerate}
        \item[(i)] If $\xt \sim \ber(1/2)$,
        then $\tilde{m} = 1/2$ and $\st^2 = 1/4$, so we can calculate, for any $z\in\N_1$,
    	$d(z) =  \frac{z - 1}{2z}$, $d(z)(1 - d(z)) = \frac{(z-1)(z+1)}{4z^2}$,
    	and $\frac{\st^2 \cdot \phi(z)}{z} = \frac{z-1}{4z}$.
    	From \thref{WhenPSDBPMustMatchDCBP}, we know that $\cbp$ can be matched by a PSDBP if and only if $\frac{z-1}{4z} \geq \frac{(z-1)(z+1)}{4z^2}$.
        For $z \geq 2$, this simplifies to $z \geq z+1$, which never holds. Hence there is no PSDBP with the same mean and variance as $\cbp$. \\

        \item[(ii)] If $\xt \sim \bin(2, 1/2)$,
        $\cbp$ has $\tilde{m} = 1$ and $\tilde{\sigma}^2 = 1/2$,
        so we can calculate $d(z) = \frac{z - 1}{z}$,
        $d(z)(1 - d(z)) = \frac{z - 1}{z^2}$, and $\frac{\tilde{\sigma}^2 \cdot \phi(z)}{z} = \frac{z-1}{2z}$.
        Appealing to \thref{WhenPSDBPMustMatchDCBP}, we can find a PSDBP to match $\cbp$ if and only if $\frac{z-1}{2z} \geq \frac{z-1}{z^2}$.
        When $z = 1$ this is trivially true, and for $z \geq 2$, it simplifies to $1/2 \geq 1/z$, which is again true.
        Hence in this case, it \textit{is} possible to find a PSDBP with matching moments to $\cbp$. \\

        The explicit form of this matching PSDBP is not unique. One possible solution is to take the PSDBP with $\xi(0) = \xi(1) = 0$,
        and, for $z \geq 2$, the three-point distribution with $\P(\xi(z) = 0) = \frac{z^2+z+2}{4z^2}$, $\P(\xi(z) = 1) = \frac{z^2+z-2}{2z^2}$,
        and $\P(\xi(z) = 2) = \frac{z^2-3z+2}{4z^2}$.
        We plot trajectories from the PSDBP with this offspring distribution against trajectories from $\cbp$
        in Figure \ref{fig:MatchingMoments}.
    \end{enumerate}

    We can consider these findings in the context of \thref{MinimumDiscreteVariance}:
    the $\ber(1/2)$ distribution achieves the minimum possible variance of any distribution
    taking values in $\N_0$ with mean $1/2$.
    Since $\phi(z) < z$ for all $z \geq 2$, by \thref{MatchingMomentsEquations}, a PSDBP with matching moments would require an offspring distribution with lower variance than $\tilde{\xi}$.
    On the other hand, the $\bin(2, 1/2)$ distribution has a higher variance than the minimum possible for a distribution on $\N_0$ with mean 1.
    There is enough `excess' left over, even after adjusting by $(z-1)/z$, for a PSDBP to match.
    \null\hfill$\blacklozenge$
\end{exmp}

\subsection{Estimation of the TVD between a PSDBP and CBP}

\thref{KStepTVDBound} provides an analytic bound on the TVD for a range of DCBPs and PSDBPs with matching moments, but this bound is not tight,
and Section \ref{sec:matchingmomentsPSDBPandDCBP} demonstrated that it is not trivial to match moments when a CBP is constrained to have a deterministic control function.

For any PSDBP-CBP pair, including those that do not meet the requirements of \thref{KStepTVDBound},
we can use an importance sampling approach to estimate the TVD between their distributions.

\begin{lemma}\label{TVDEstimator}
    Let $X$ and $Y$ be two random variables defined on a countable space $\mathcal{X}$.
    Let $\mathbf{x}_N := (x_1, \dots, x_N)$ be a sample of $N$ independent observations drawn from the distribution of $X$, $\L_X$. Then
    \[
        \hat{\theta}_N := \frac{1}{2N} \sum_{i=1}^N \frac{| \L_X(x_i) - \L_Y(x_i) |}{\L_X(x_i)}
    \]
    is an unbiased, consistent estimator for $|| \L_X - \L_Y ||_{TV}$.
\end{lemma}

Hence, for $k \in \N_1$, if we can simulate a large number of trajectories from the $k$-step distribution of a CBP or PSDBP,
\thref{TVDEstimator} provides us with a method to estimate the TVD between the two processes.
Algorithm \ref{alg:TVDAlgorithm} uses this method to produce an estimate for the TVD between a CBP $\cbp$ and a PSDBP $\psdbp$.

\begin{algorithm}
\caption{An estimator for the TVD}
\label{alg:TVDAlgorithm}
\begin{algorithmic}
    \State $S_1 \gets 0$
    \For{ $i \in [\text{numTrials}]$ }
        \State $(z_0, \dots, z_k) \gets \textsc{samplePath}(\{Z_n\}_{n\in[k]})$ \Comment{$k$-step path sampled from the PSDBP}

        \State $\mathcal{L}_Z \gets \textsc{likelihood}\big((z_0, \dots, z_k), \{Z_n\}_{n\in[k]}\big)$
        \State $\mathcal{L}_{\zt} \gets \textsc{likelihood}\big((z_0, \dots, z_k), \{\zt_n\}_{n\in[k]}\big)$

        \State $x \gets | \mathcal{L}_Z - \mathcal{L}_{\zt} | \mathbin{/} 2 \mathcal{L}_Z$

        \State $S_1 \gets S_1 + x$
    \EndFor
    \State \Return $S_1 \mathbin{/} \text{numTrials}$
\end{algorithmic}
\end{algorithm}

While in Algorithm \ref{alg:TVDAlgorithm} we chose to sample from the PSDBP,
the symmetry of the TVD means that  a similar result holds if we instead sampled from the CBP. \\

\thref{KStepTVDBound} established a TVD bound that decreases as $z^{-1/2}$.
We use Algorithm \ref{alg:TVDAlgorithm} to demonstrate similar decaying behaviour
in a PSDBP-CBP pair beyond the scope of \thref{KStepTVDBound}:
two processes which grow logistically towards and fluctuate around a carrying capacity.

\begin{figure}[!b]
	\centering
	\includegraphics[width=13cm]{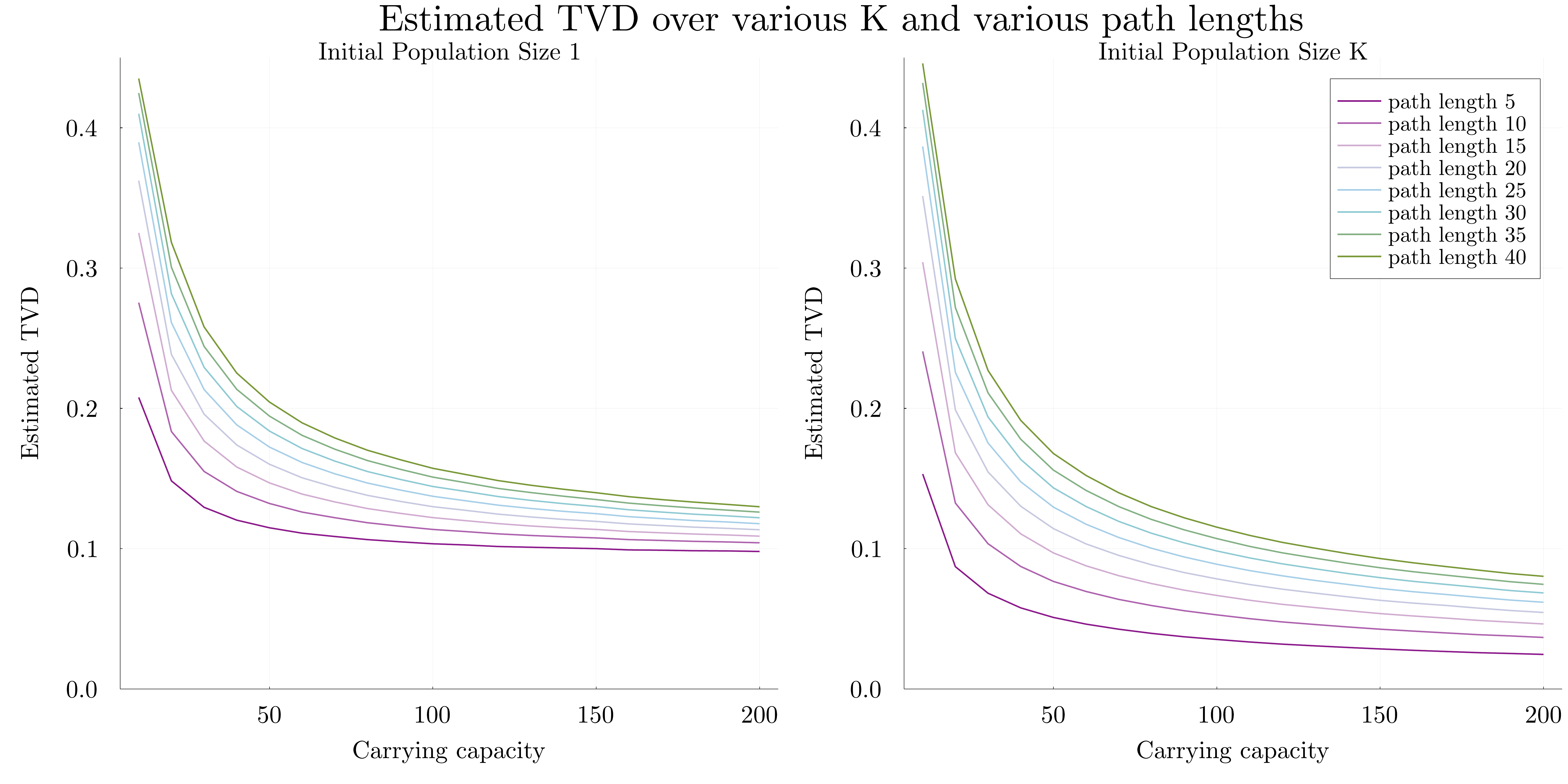}
	\caption{Estimated TVD between the $k$-step distributions of the processes $\{\tilde{Z}_n\}$ and $\{Z_n\}$ of Example 4.15, 
    for $K \in [200]$ and various path lengths,
    taking the initial population size of both processes as 1 (left) and K (right).
	Each estimated value is the average of 1,000,000 iterations.}
    \label{fig:PoissonKPathLength1}
	\vspace{0.7cm}
\end{figure}

\begin{exmp}\label{CarryingCapacityTVDBound}
    For $\lambda \geq 2$, $M, K \in \N_1$, $M < K$, consider the CBP $\cbp$ with
    \[
        \pt(z) \sim \bin\left((z + M)\1_{\{z>0\}},\; \frac{2K^2}{\lambda(K+M) (z+K)}\right) \quad \text{and} \quad
    	\xt \sim \poi\left( \lambda \right),
    \]
    and the PSDBP $\psdbp$ with
    \[
    	\xi(z) \sim \nb\left(
    	   \frac{2K^2(z+M)}{z(\lambda(K+M)z + \lambda KM +(\lambda-2)K^2)},\;
    	   \frac{(z+K)(K+M)}{(1+\lambda)(K+M)(K+z)-2K^2}
        \right).
    \]
    One can check that both processes have matching moments, and a carrying capacity at $K$: we have that $\E(Z_n | Z_{n-1} = K) = \E(\zt_n | \zt_{n-1} = K) = K$,
    $\E(Z_n | Z_{n-1} = x) = \E(\zt_n | \zt_{n-1} = x) > x$ for $x < K$,
    and $\E(Z_n | Z_{n-1} = y) = \E(\zt_n | \zt_{n-1} = y) < y$ for $y > K$. \\
    
    Taking $\lambda = 3$ and $M = 2$, we produce Figure \ref{fig:PoissonKPathLength1}.
    As may be expected, we see that the TVD between the two processes increases in the path length, and decreases in the carrying capacity.
    The TVD becomes quite small at even moderate values of $K$, especially in the case where $z_0 = K$,
    suggesting that one would be indifferent to the choice of either $\cbp$ or $\psdbp$ to model a population.

    It is also noteworthy that while the processes with $z_0 = K$ have a TVD that appears to approach zero,
    the processes with $z_0 = 1$ do not --- the PSDBP and CBP are at their most different at small population sizes,
    so if the processes begin from a small population, some `accumulated TVD' will necessarily build up.
    \null\hfill$\blacklozenge$
\end{exmp}

\section{Conclusion}\label{sec:conclusion}

Population-size-dependent branching processes and controlled branching processes are both well-used in the field of population biology,
providing a simple yet flexible way to model populations amid phenomena such as resource scarcity and immigration.
While PSDBPs lack the ability to incorporate external random factors, CBPs can via the control function.

However, we have shown that many CBP models considered in the literature, including those with Poisson, negative binomial, and binomial control functions,
can be represented as PSDBPs, either exactly or approximately.
Therefore, when considering models with these common control functions, no additional benefit is gained compared to using a PSDBP.

In the presence of environmental stochasticity, which class of models to focus on and how to estimate their parameters
become important questions. Determining the consistent estimation of parameters for these models is an ongoing area of research.

\section{Proofs}\label{sec:proofs}

\subsection{Proofs from Section \ref{sec:EquivalentBranchingProcesses}}

\begin{proof}[\textbf{Proof of \protect{\thref{PoissonOffspringNoEquivalence}}.}]
	Firstly, if $\psi$ is such that $\psi(0) > 0$, the CBP allows immigration, and we know trivially that there will be no equivalent PSDBP.
	Assume, then, that $\psi(0) = 0$.
	We also assume that $\psi(z) \geq 1$ for all $z \geq 1$ to exclude trivial equivalences: otherwise, with $z_0 = 1$ and $\psi(1) = 0$,
	$\cbp$ would be equivalent to the PSDBP with $\xi(z) = 0$ a.s.\ for all $z\in\N_0$.
	
	In addition, under the assumption that $\pt(z)$ is not $\zt$-divisible,
	there exists an attainable $z^*\in\N_0$ such that $\pt(z^*)$ is not $z^*$-divisible,
	i.e.\ $\psi(z^*)$ is not a multiple of $z^*$. In this case, using the composition property of probability generating functions (PGFs) for a random sum,
	the PGF of $(\zt_n | \zt_{n-1} = z^*) = \sum_{i=1}^{\pt(z^*)} \xt_i$ is given by
    \[
        G(t) = \left( p(z^*) + (1-p(z^*))e^{\lambda(t-1)} \right)^{\psi(z^*)}.
    \]
	To show that $(\zt_n | \zt_{n-1} = z^*)$ is not $z^*$-divisible, we show that $\sqrt[z^*]{G(t)}$ is not a valid PGF, and make use of the next lemma and proposition.
\end{proof}

\begin{lemma}\label{DiscreteDistributionsDivideIntoDiscreteDistributions}
    Let $X$ be a discrete random variable supported on the non-negative integers.
	Then, for any $z\in\N_1 \setminus \{1\}$ such that $X$ is $z$-divisible, with $X \stackrel{d}{=} \sum_{i=1}^n X_i^{(z)}$, each $X_i^{(z)}$ is also a discrete random variable supported on the non-negative integers.
\end{lemma}

\begin{proof}
	Suppose that $X$ is a discrete random variable, and suppose that $X$ is $z$-divisible, for some $z\in\N_1 \setminus \{1\}$.
	Then $X \stackrel{d}{=} \sum_{i=1}^z X_i^{(z)}$, for $X_1^{(z)}, \dots, X_z^{(z)}$ i.i.d..
	
	It is a classical result that discrete distributions can only be decomposed into discrete distributions
	(see Corollary 4 to Theorem 3.2.1 in \cite{linnik77}), so $X_1^{(z)}$ must have a discrete distribution.

	Assume that $\P\big( X_1^{(z)} < 0 \big) > 0$.
	Then
	\[
		\P(X < 0) = \P\left( \sum_{i=1}^z X_i^{(z)} < 0 \right) \geq \prod_{i=1}^z \P\left( X_i^{(z)} < 0 \right) > 0,
	\]
	but by assumption $X$ has a non-negative support.
	This is a contradiction, so $X_1^{(z)}$ must be supported on the non-negative integers.
    \qed
\end{proof}

\begin{proposition}[Pringsheim's Theorem]\label{PringsheimsTheorem}
	Let $G(t) = \sum_{n=0}^{\infty} a_n t^n$ be a power series with a radius of convergence $R > 0$.
	Then if $a_n \geq 0$ for all $n \in \N_0$, the singularity of $G(t)$ closest to the origin is at the point $t=R$ on the real line.
\end{proposition}

Pringsheim's Theorem is a well-known result in complex analysis,
and proofs can be found in \cite[Theorem 7.2]{titchmarsh39} and \cite[Theorem IV.6]{flajolet09}. \\

\begin{proof}[Proof of \thref{PoissonOffspringNoEquivalence}, continued.]
    Suppose that $(\zt_n | \zt_{n-1} = z^*)$ is $z^*$-divisible, with $(\zt_n | \zt_{n-1} = z^*) \stackrel{d}{=} \sum_{i=1}^{z^*} \zeta_i^{(z^*)}$.
	Since the distribution of $(\zt_n | \zt_{n-1} = z^*)$ is discrete and supported on the non-negative integers,
	then from \thref{DiscreteDistributionsDivideIntoDiscreteDistributions} we know that the distribution
	of the $\zeta_1^{(z^*)}$ must also be discrete and supported on the non-negative integers.
    Accordingly, $\zeta_1^{(z^*)}$ must have a PGF, and this PGF must correspond to $\sqrt[z^*]{G(t)}$.
	Hence, if we can show that $\sqrt[z^*]{G(t)}$ is not a valid PGF, we are done.
	
	Definitionally, PGFs are power series with non-negative coefficients,
	so \thref{PringsheimsTheorem} provides a test which, if not met, allows us to reject candidate PGFs as invalid.
	That is to say, by \thref{PringsheimsTheorem}, if  $\sqrt[z^*]{G(t)}$ has a singularity \textit{off} the positive real half-line that is closer to the origin
	than the closest singularity \textit{on} the positive real half-line, then $\sqrt[z^*]{G(t)}$ is not a valid PGF. We have that
	\[
		\sqrt[z^*]{G(t)} = \left( p(z^*) + (1-p(z^*))e^{\lambda(t-1)} \right)^{\psi(z^*) /z^*},
	\]
	and since by assumption $\psi(z^*)$ is not a multiple of $z^*$, $\psi(z^*) / z^*$ is not an integer.
 
    We define the function $f(t) := p(z^*) + (1-p(z^*))e^{\lambda(t-1)}$, such that $\sqrt[z^*]{G(t)} = f(t)^{\psi(z^*) /z^*}$.
    Since $f(t)$ is an entire function, we observe that the only singularities of $\sqrt[z^*]{G(t)}$ are the branch points that occur whenever $f(t) = 0$.
	
	But $f(t) = 0 \implies e^{\lambda(t-1)} = \frac{p(z^*)}{p(z^*) - 1}$, which is solved at
	\begin{align*}
		t &= 1 + \frac{1}{\lambda} \log\left( \frac{p(z^*)}{p(z^*) - 1} \right) \\
			&= 1 + \frac{1}{\lambda} \left( \log\left|\frac{p(z^*)}{p(z^*) - 1}\right|
				+ i\left( \arg\left( \frac{p(z^*)}{p(z^*) - 1} \right) + 2\pi\Z \right) \right) \\
			&= 1 + \frac{1}{\lambda} \log\left( \frac{p(z^*)}{1 - p(z^*)} \right) + \frac{i\pi}{\lambda}(1+2\Z).
	\end{align*}
	That is, the branch points of $\sqrt[z^*]{G(t)}$ all occur off the real axis,
	and therefore $\sqrt[z^*]{G(t)}$ is not a valid PGF.
	Hence $(\zt_n | \zt_{n-1} = z^*)$ is not $z^*$-divisible,
	and consequently $\cbp$ cannot be expressed equivalently as a PSDBP.
    \qed
\end{proof}
\newline

\begin{proof}[\textbf{Proof of \thref{GeometricOffspringEquivalence}.}]
    The proof of this result relies on the following lemma.
\end{proof}

\begin{lemma}\label{ZIGInfinitelyDivisible}
	Suppose that $X$ has a zero-inflated geometric distribution with parameters $p \in (0,1)$ and $q \in (0,1)$, i.e.\ $X \sim \zig(p,q)$,
	so that $X \stackrel{d}{=} YZ$ for $Y \sim \ber(1-p)$ and $Z \sim \geom(q)$ independent.
	Then $X$ is infinitely divisible.
\end{lemma}
\begin{proof}
	Assume that $X$ is a random variable with a discrete distribution supported on the non-negative integers.
	Warde and Katti show in Theorem 2.1 of \cite{warde71} that if $P(X=0) > 0$, $P(X=1) > 0$,
	and $\big\{ \P(X=k+1) / \P(X=k) \big\}_{k\in\N_0}$ is a non-decreasing sequence
	(in \cite{steutel79} and \cite[p.~59--65]{steutel03} this condition is called log-convexity), then $X$ is infinitely divisible.
	
	Suppose that $X \sim \zig(p,q)$, for $p \in (0,1)$ and $q \in (0,1)$.
	Then $P(X = 0) = p + (1-p)q$ and $P(X = k) = (1-p)(1-q)^k q$ for $k \geq 1$.
	Clearly $P(X = 0) > 0$ and $P(X = 1) > 0$.
	In addition,
	\[
		\frac{\P(X = 1)}{\P(X = 0)} = (1-q) \frac{(1-p)q}{p + (1-p)q} < 1-q,
	\]
	and for $k \geq 1$,
	\[
		\frac{\P(X=k+1)}{\P(X=k)} = 1-q,
	\]
	so that $\big\{ \P(X=k+1) / \P(X=k) \big\}_{k\in\N_0}$ forms a non-decreasing sequence.
    \qed
\end{proof}
\newline

\begin{proof}[Proof of \thref{GeometricOffspringEquivalence}, continued.]
	Assume that $\cbp$ is as stated above, with $\psi(0) = 0$. Then for $z \in \N_1$,
	\[
		(\zt_{n+1} | \zt_n = z)
			\stackrel{d}{=} \sum_{i=1}^{\pt(z)} \xt_i
			\stackrel{d}{=} \sum_{i=1}^{\psi(z)} B_i(z) \cdot \xt_i
			\stackrel{d}{=} \sum_{i=1}^{\psi(z)} \zeta_i(z),
	\]
	where $B_i(z) \sim \ber(p(z))$, and $\zeta_i (z):= B_i(z) \cdot \xt_i \sim \zig(1-p(z),q)$.
	By \thref{ZIGInfinitelyDivisible} the $\zeta_i(z)$'s are infinitely divisible,
	so we can represent each as the sum of $z$ i.i.d.\ random variables, say $\{\chi_{i,j}(z)\}_{j\in\{1,\dots,z\}}$, to get
	\[
		\sum_{i=1}^{\psi(z)} \zeta_i(z)
			\stackrel{d}{=} \sum_{i=1}^{\psi(z)} \sum_{j=1}^z \chi_{i,j}(z)
			\stackrel{d}{=} \sum_{i=1}^z \sum_{j=1}^{\psi(z)} \chi_{i,j}(z)
			\stackrel{d}{=} \sum_{i=1}^z \xi_i(z),
	\]
	for $\xi_i(z) := \sum_{j=1}^{\psi(z)} \chi_{i,j}(z)$.
	But this form is just that of a PSDBP.
    \qed
\end{proof}
\newline

\begin{proof}[\textbf{Proof of \thref{DCBPEquivalenceConditions}.}]
	Let $\cbp$ be a DCBP with control function $\phi(\cdot)$, offspring distribution $\xt$, and initial population size $z_0 \in \N_0$.
	In the case that $\phi(0) \neq 0$, by \thref{noimmigration} $\cbp$ cannot be expressed as a PSDBP.
	Therefore to prove the result it remains to show that, under the assumption that $\phi(0) = 0$,
	$\cbp$ can be expressed as a PDSBP if and only if $\cbp$ is $\zt$-divisible.
	
	We prove the forward direction first: assume that $\cbp$ is $\zt$-divisible (which implies that $\phi(0) = 0$).
	We want to show that there exist appropriate i.i.d.\ random variables $\xi_{n,i}(z)$ such that,
	for all attainable $z \neq 0$,
	$\sum_{i=1}^{\phi(z)} \tilde{\xi}_{n,i} \stackrel{d}{=} \sum_{l=1}^z \xi_{n,l}(z)$.

	But since $\cbp$ is $\tilde{Z}$-divisible, there exists $y\in\N_1$ such that $x := y \cdot \phi(z) / z \in \N_0$
	and $\tilde{\xi}_{n,i} \stackrel{d}{=} \sum_{j=1}^{y} \zeta_{n,i,j}^{(y)}$ for some i.i.d.\ random variables $\zeta_{n,i,j}^{(y)}$.
	Therefore
	\begin{equation*}
		\sum_{i=1}^{\phi(z)} \tilde{\xi}_{n,i}
		\stackrel{d}{=} \sum_{i=1}^{\phi(z)} \sum_{j=1}^y \zeta_{n,i,j}^{(y)}
		\stackrel{d}{=} \sum_{k=1}^{\phi(z)\cdot y} \zeta_{n,k}^{(y)}
		\stackrel{d}{=} \sum_{l=1}^z \xi_{n,l}(z)
	\end{equation*}
	where the $\zeta_{n,k}^{(y)}$'s are just reindexed $\zeta_{n,i,j}^{(y)}$'s (in any order, since they are i.i.d.),
	and $\xi_{n,l}(z) := \sum_{k=x(l-1) + 1}^{xl} \zeta_{n,k}^{(y)}$, $l \in [z]$.

    In the reverse direction, we first need to show the following intermediary result.
\end{proof}

\begin{lemma}\label{divisiblefactors}
	Let $X$ be a random variable. If, for some $p\in\N_1\setminus\{1\}$, $X$ is $p$-divisible,
	then $X$ is also $q$-divisible for any $q\in\N_1\setminus\{1\}$ such that $q$ divides $p$.
\end{lemma}

\begin{proof}
	If $q$ divides $p$ then there exists an $r\in\N_1$ such that $p = rq$. Also, since $X$ is $p$-divisible, there exist i.i.d.\ RVs
	$X_1^{(p)}, \dots, X_p^{(p)}$ such that $X \stackrel{d}{=} \sum_{i=1}^p X_i^{(p)}$.
	Therefore
	\begin{equation*}
		X \stackrel{d}{=} \sum_{i=1}^p X_i^{(p)}
		= \sum_{i=1}^{rq} X_i^{(p)}
		= \sum_{j=1}^{q} \sum_{k=r(j-1)+1}^{rj} X_k^{(p)}
		= \sum_{j=1}^{q} Y_j,
	\end{equation*}
	where $Y_j := \sum_{k=r(j-1)+1}^{rj} X_k^{(p)} \stackrel{d}{=} \sum_{k=1}^{r} X_k^{(p)}$.
	Since the $Y_j$'s are i.i.d., the definition of $q$-divisibility is satisfied.
    \qed
\end{proof}
\newline

\begin{proof}[Proof of \thref{DCBPEquivalenceConditions}, continued.]
	Assume that $\phi(0) = 0$ but $\cbp$ is \textit{not} $\tilde{Z}$-divisible.
	Then there exists a $y^* \in \mathcal{Y}_{\tilde{Z}}$ for which there is \textit{no} collection $\big\{ \zeta_{n,i,j}^{(y^*)} \big\}$
	of i.i.d.\ random variables such that $\tilde{\xi}_{n,i} \stackrel{d}{=} \sum_{j=1}^{y^*} \zeta_{n,i,j}^{(y^*)}$.
	But, by \thref{divisiblefactors}, if $\cbp$  is not $y^*$-divisible then for each
	$z \in \mathcal{A}(z_0) \setminus \{0\}$ such that
	$z / \text{gcd}[\phi(z), z] = y^*$, $\sum_{i=1}^{\phi(z)} \tilde{\xi}_{n,i}$ is not $z$-divisible.
	Therefore there are no $\xi_{n,l}(z)$'s such that $\sum_{i=1}^{\phi(z)} \tilde{\xi}_{n,i} \stackrel{d}{=} \sum_{l=1}^z \xi_{n,l}(z)$.
    \qed
\end{proof}
\newline

\subsection{Proofs from Section \ref{sec:approxequal}}

\begin{definition}
	A random variable $W$ is said to have a \textit{\textbf{discretised normal distribution}} 
	with parameters $\mu$ and $\sigma^2$ if for every $k\in\Z$,
	\[
		\P(W = k)
		= \frac{1}{\sqrt{2\pi\sigma^2}} \int_{k - \frac{1}{2}}^{k + \frac{1}{2}} e^{-\frac{(u - \mu)^2}{2\sigma^2}} du.
	\]
\end{definition}

\begin{proof}[{\bfseries Proof of \thref{OneStepTVDBound}}.]
	Since $\psdbp$ and $\cbp$ are assumed to have matching first and second moments,
    it follows from \thref{MatchingMomentsEquations} that $z \cdot m(z) = \tilde{m} \cdot \phi(z)$ and $z \cdot \sigma^2(z) = \tilde{\sigma}^2 \cdot \phi(z)$.
    Let $W$ be a random variable with a discretised normal distribution with parameters $z \cdot m(z)$ and $z \cdot \sigma^2(z)$.
    We can then use the triangle inequality to obtain the bound
	\[
		|| \L_{Z_1|Z_0 = z} - \L_{\tilde{Z}_1|\tilde{Z}_0 = z} ||_{TV}
			\leq || \L_{Z_1|Z_0 = z} - \L_W ||_{TV}
			+ || \L_{\tilde{Z}_1|\tilde{Z}_0 = z} - \L_W ||_{TV}.
	\]
    Comparison to the discretised normal allows us to leverage the results of \cite{chen11} to produce a closed-form expression for the TVD bound via the next lemma.
 \end{proof}

\begin{lemma}\label{DiscretisedNormalTVDBound}
	For $n\in\N_1$, let $X_1, \dots, X_n$ and $X$ be i.i.d.\ random variables on $\Z$,
	with $\E X := \mu$, $\V(X) := \sigma^2$, and finite third absolute central moment $\rho := \E|X - \mu|^3$.
    Let $W_n$ be a random variable with a discretised normal distribution with parameters $n\mu$ and $n\sigma^2$,
    and define $S_n := \sum_{i=1}^n X_i$. Then we have
	\begin{align*}
		||\L_{S_n} - \L_{\bar{W}} ||_{TV} &\leq \sqrt{\frac{2}{\pi}} \left( \frac{3\rho}{\sigma^2} + 2 \right)
			\Big( 1 + 4 (n-1)\left( 1- ||\L_{X} - \L_{(X + 1)} ||_{TV} \right) \Big)^{-\frac{1}{2}} \\
		&\quad + \left( 5 + 3\sqrt{\frac{\pi}{8}} \right)\frac{\rho}{\sqrt{n}\sigma^3} + \frac{1}{2\sqrt{2\pi n}\sigma}.
	\end{align*}
\end{lemma}

\begin{proof}
	In the case of an i.i.d.\ sum, Theorem 7.4 of \cite{chen11} simplifies to
	\begin{equation*}
		||\L_{S_n} - \L_{W_n} ||_{TV}
		\leq \left( \frac{3\rho}{2\sigma^2} + 1 \right) || \L_{S_{n-1}} - \L_{(S_{n-1} + 1)} ||_{TV}
		+ \left( 5 + 3\sqrt{\frac{\pi}{8}} \right)\frac{\rho}{\sqrt{n}\sigma^3}
		+ \frac{1}{2\sqrt{2\pi n}\sigma}.
	\end{equation*}
	Corollary 1.6 of \cite{mattner07} then provides the bound
	\begin{align*}
		||\L_{S_{n-1}} - \L_{(S_{n-1} + 1)} ||_{TV}
		&\leq \sqrt{\frac{2}{\pi}} \left( \frac{1}{4} + (n-1)\left( 1 - ||\L_{X} - \L_{(X + 1)} ||_{TV} \right) \right)^{-\frac{1}{2}} \\
		&= 2 \sqrt{\frac{2}{\pi}} \Big( 1 + 4 (n-1)\left( 1 - ||\L_{X} - \L_{(X + 1)} ||_{TV} \right) \Big)^{-\frac{1}{2}},
	\end{align*}
	which yields the desired result.
    \qed
\end{proof}
\newline

\begin{proof}[Proof of \thref{OneStepTVDBound}, continued.]
    Since Condition (C2) ensures that the third absolute central moments of both $\psdbp$ and $\cbp$ are finite,
    applying \thref{DiscretisedNormalTVDBound} to our setting allows us to obtain
	\begin{align*}
		& || \L_{Z_1|Z_0 = z} - \L_{\tilde{Z}_1|\tilde{Z}_0 = z} ||_{TV} \\
		&\leq \sqrt{\frac{2}{\pi}} \left( \frac{3\rho(z)}{\sigma(z)^2} + 2 \right)
			\Big( 1 + 4 (z-1)\left( 1 - ||\L_{\xi(z)} - \L_{(\xi(z) + 1)} ||_{TV} \right) \Big)^{-\frac{1}{2}} \\
		&\quad + \sqrt{\frac{2}{\pi}} \left( \frac{3\tilde{\rho}}{\tilde{\sigma}^2} + 2 \right)
			\Big( 1 + 4(\phi(z)-1)\left( 1 - ||\L_{\tilde{\xi}} - \L_{(\tilde{\xi} + 1)} ||_{TV} \right) \Big)^{-\frac{1}{2}} \\
		&\quad + \left( 5 + 3\sqrt{\frac{\pi}{8}} \right)\left( \frac{\rho(z)}{\sqrt{z}\sigma^3(z)}
			+ \frac{\tilde{\rho}}{\sqrt{\phi(z)}\tilde{\sigma}^3} \right)
			+ \frac{1}{\sqrt{2\pi \phi(z)}\tilde{\sigma}}.
	\end{align*}

	We next want to remove the dependence of this bound on $||\L_{\xi(z)} - \L_{(\xi(z) + 1)} ||_{TV}$ and
	$||\L_{\tilde{\xi}} - \L_{(\tilde{\xi} + 1)} ||_{TV}$.
	Indeed, the bound is \textit{increasing} in both $||\L_{\xi(z)} - \L_{(\xi(z) + 1)} ||_{TV}$
	and $||\L_{\tilde{\xi}} - \L_{(\tilde{\xi} + 1)} ||_{TV}$, so we want to find a constant $\gamma$ such that
	\[
		\gamma \geq ||\L_{\tilde{\xi}} - \L_{(\tilde{\xi} + 1)} ||_{TV} \text{   and   }
		\gamma \geq ||\L_{\xi(z)} - \L_{(\xi(z) + 1)} ||_{TV} \text{ for all } z\in\N_1,
	\]
	or equivalently a constant $\delta$ such that
	\[
		\delta \leq 1 - ||\L_{\tilde{\xi}} - \L_{(\tilde{\xi} + 1)} ||_{TV} \text{   and   }
		\delta \leq 1 - ||\L_{\xi(z)} - \L_{(\xi(z) + 1)} ||_{TV} \text{ for all } z\in\N_1.
	\]

	We can re-express $1 - ||\L_{\tilde{\xi}} - \L_{(\tilde{\xi} + 1)} ||_{TV}$ as $\sum_{i\in\N_0} \P(\tilde{\xi} = i) \wedge \P(\tilde{\xi} = i - 1)$ by using \eqref{eqn:TVDOneMinusEquivalentSum},
	so it follows that $\P(\tilde{\xi} = n) \wedge \P(\tilde{\xi} = n - 1) \leq 1 - ||\L_{\tilde{\xi}} - \L_{(\tilde{\xi} + 1)} ||_{TV}$
	for any fixed $n\in\N_1$. An analogous result holds for $1 - ||\L_{\xi(z)} - \L_{(\xi(z) + 1)} ||_{TV}$.
	
	Hence, if such an $\eta$ as in Condition (C3) exists, then $2\eta \leq 1 - ||\L_{\tilde{\xi}} - \L_{(\tilde{\xi} + 1)} ||_{TV}$
	and $2\eta \leq 1 - ||\L_{\xi(z)} - \L_{(\xi(z) + 1)} ||_{TV}$ for every $z\in\N_1$, so we can take $\delta \equiv \eta$.
	Alongside Conditions (C1) and (C2), this allows us to simplify the bound to
	\begin{align*}
		& || \L_{Z_1|Z_0 = z} - \L_{\tilde{Z}_1|\tilde{Z}_0 = z} ||_{TV} \\
		&\leq \sqrt{\frac{2}{\pi}} \left( \frac{3R}{h\tilde{\sigma}^2} + 2 \right) \left( 1 + 4\eta(z-1) \right)^{-\frac{1}{2}}
			+ \sqrt{\frac{2}{\pi}} \left( \frac{3R}{\tilde{\sigma}^2} + 2 \right) \left( 1 + 4\eta(hz - 1) \right)^{-\frac{1}{2}} \\
		&\quad +  \left( 5 + 3\sqrt{\frac{\pi}{8}} \right) \left( \frac{R}{\sqrt{zh^3}\tilde{\sigma}^3} + \frac{R}{\sqrt{zh}\tilde{\sigma}^3} \right)
			+ \frac{1}{\sqrt{2\pi hz}\tilde{\sigma}} \\
		&\leq \frac{2\sqrt{2}\left( 3R + 2(1 + h)\tilde{\sigma}^2 \right)}{\tilde{\sigma}^2 (h\wedge 1)\sqrt{\pi + 4\pi\eta(z - 1)}}
			+ \frac{\left( 5\sqrt{2\pi} + \frac{3\pi}{2} \right) (1 + h) R + h\tilde{\sigma}^2}{\tilde{\sigma}^3\sqrt{2\pi h^3 z}}.
	\end{align*}
 
	It remains to show that $\sqrt{\pi + 4\pi\eta(z - 1)} \geq  2\sqrt{\pi\eta z}$.
	Because $\P(\tilde{\xi} = n) + \P(\tilde{\xi} = n - 1) \leq 1$, it must be the case that
	$\P(\tilde{\xi} = n) \wedge \P(\tilde{\xi} = n - 1) \leq \frac{1}{2}$ for every $n \in \N_1$.
	Hence it follows from Condition (C3) that $\eta \leq \frac{1}{4}$, so that
	\[
		\sqrt{\pi + 4\pi\eta(z - 1)} = \sqrt{\pi - 4\pi\eta + 4\pi\eta z} \geq \sqrt{\pi - \pi + 4\pi\eta z} = 2\sqrt{\pi\eta z},
	\]
	which implies that
	\[
		|| \L_{Z_1|Z_0 = z} - \L_{\tilde{Z}_1|\tilde{Z}_0 = z} ||_{TV}
		\leq \frac{\sqrt{2}\left( 3R + 2(1 + h)\tilde{\sigma}^2 \right)}{\tilde{\sigma}^2 (h\wedge 1) \sqrt{\pi \eta z}}
			+ \frac{\left( 5\sqrt{2\pi} + \frac{3\pi}{2} \right) (1 + h) R + h\tilde{\sigma}^2}{\tilde{\sigma}^3\sqrt{2\pi h^3 z}},
	\]
	our desired bound.
    \qed
\end{proof}
\newline

\begin{proof}[\textbf{Proof of \thref{KStepTVDBound}}.]
	Suppose that we have a PSDBP $\psdbp$ and a DCBP $\cbp$
	with matching moments and with $Z_0 = \tilde{Z}_0 = z \in \N_1$,
    such that Conditions (C1), (C2), and (C3) are satisfied.
	Fix an $\alpha \in (0,1)$ and $k\in\N_1$.
 
    Assume that for any $j \in \N_1$ such that $j \leq k$,
	\begin{equation}\label{eqn:TVDInductionBound}
		||\L_{(Z_j, \dots, Z_k) | Z_{j-1} = u_{j-1}} - \L_{(\tilde{Z}_j, \dots, \tilde{Z}_k) | \tilde{Z}_{j-1} = u_{j-1}}||_{TV}
		\leq \K_j(u_{j-1}),
	\end{equation}
	where $\K_j: \R_{\geq 0} \to \R_{\geq 0}$ is a monotonically decreasing function of $u_{j-1}$ given by $\K_k(u_k) := \J(u_k)$ and for $j < k$,
	\begin{equation}\label{eqn:KRecursiveForm}
		\K_j(u_{j-1}) := \J(u_{j-1}) + \frac{\tilde{\sigma}^2}{(1-\alpha)^2\tilde{m}^2 h \cdot u_{j-1}}
			+ \K_{j+1}\big( \alpha\tilde{m}h \cdot u_{j-1} \big).
	\end{equation}
    Then, taking $u_0 := z$, we can apply \eqref{eqn:TVDInductionBound} iteratively $k$ times to obtain
	\begin{align*}
		&||\L_{(Z_1, \dots, Z_k)|Z_0 = z} - \L_{(\tilde{Z}_1, \dots, \tilde{Z}_k)|\tilde{Z}_0 = z}||_{TV} \\
		&\leq \J(z) + \frac{\tilde{\sigma}^2}{(1-\alpha)^2\tilde{m}^2 h \cdot z} + \K_2\big( \alpha\tilde{m}h \cdot z \big) \\
		&\leq \J(z) + \J\big( \alpha\tilde{m}h \cdot z \big)
			+ \frac{\tilde{\sigma}^2}{(1-\alpha)^2\tilde{m}^2 h \cdot z} + \frac{\tilde{\sigma}^2}{(1-\alpha)^2 \alpha \tilde{m}^3 h^2 \cdot z}
			+ \K_3\big( (\alpha\tilde{m}h)^2 \cdot z \big) \\
		&\vdots \\
		&\leq \sum_{i=0}^{k-2} \J\big( ( \alpha \tilde{m} h )^i \cdot z \big)
			+ \frac{\tilde{\sigma}^2}{(1 - \alpha)^2 \tilde{m}^2 h \cdot z} \cdot \sum_{i=0}^{k-2} ( \alpha \tilde{m} h )^{-i}
			+ \K_k\big( ( \alpha \tilde{m} h )^{k-1} \cdot z \big) \\
		&= \sum_{i=0}^{k-1} \J\big( ( \alpha \tilde{m} h )^i \cdot z \big)
			+ \frac{\tilde{\sigma}^2}{(1 - \alpha)^2 \tilde{m}^2 h \cdot z} \cdot \sum_{i=0}^{k-2} ( \alpha \tilde{m} h )^{-i},
	\end{align*}
	which is what we intended to show. \\

    Hence, it remains to prove that the above assumption is true,
    which is to say that the bound \eqref{eqn:TVDInductionBound} does hold for all $j \leq k$,
    where $\K_j$ is a decreasing function given by $\K_k := \J$ and by \eqref{eqn:KRecursiveForm} for $j < k$. \\

	\textit{Base case}:
	Under the assumed conditions, and for fixed $u_{k-1}\in\N_1$, we know from \thref{OneStepTVDBound}
	(alongside the time-homogeneity of branching processes) that
	\[
		||\L_{Z_k|Z_{k-1} = u_{k-1}} - \L_{\tilde{Z}_k|\tilde{Z}_{k-1} = u_{k-1}}||_{TV} \leq \J(u_{k-1}),
	\]
	where $\J(u_{k-1})$ is a decreasing function of $u_{k-1}$. \\

	\textit{Induction step}: For $j \in \N_1$ such that $j < k$, assume that
 	\[
		||\L_{(Z_{j+1}, \dots, Z_k) | Z_j = u_j} - \L_{(\tilde{Z}_{j+1}, \dots, \tilde{Z}_k) | \tilde{Z}_j = u_j}||_{TV}
		\leq \K_{j+1}(u_j),
	\]
    where $u_j \in \N_1$ and $\K_{j+1}: \R_{\geq 0} \to \R_{\geq 0}$ is a decreasing function of $u_j$.
    We want to show that
	\[
		||\L_{(Z_j, \dots, Z_k) | Z_{j-1} = u_{j-1}} - \L_{(\tilde{Z}_j, \dots, \tilde{Z}_k) | \tilde{Z}_{j-1} = u_{j-1}}||_{TV}
		\leq \K_j(u_{j-1}),
	\]
	where $u_{j-1} \in \N_1$ and $\K_j: \R_{\geq 0} \to \R_{\geq 0}$ is a decreasing function of $u_{j-1}$ given by \eqref{eqn:KRecursiveForm}. \\

	For clarity of exposition, we introduce the notation $p_{Z_k}(a, b) := \P(Z_k = b | Z_{k-1} = a)$.
	Then, by writing the total variation distance in the form of \eqref{eqn:TVDEquivalentSum} and using the triangle inequality, we see that
    \begingroup
    \allowdisplaybreaks
	\begin{align*}
		& ||\L_{(Z_j, \dots, Z_k) | Z_{j-1} = u_{j-1}} - \L_{(\tilde{Z}_j, \dots, \tilde{Z}_k) | \tilde{Z}_{j-1} = u_{j-1}}||_{TV} \\
		&= \frac{1}{2} \sum_{u_j,\dots,u_k\geq 0}
			\bigg| \prod_{i=j}^k p_{Z_i}(u_{i-1},u_i) - \prod_{i=j}^{k} p_{\tilde{Z}_i}(u_{i-1},u_i) \bigg| \\
		&\leq \frac{1}{2}\sum_{u_j,\dots,u_k\geq 0} \bigg| \prod_{i=j}^k p_{Z_i}(u_{i-1},u_i)
			- p_{\tilde{Z}_j}(u_{j-1},u_j) \prod_{i=j+1}^{k} p_{Z_i}(u_{i-1},u_i) \bigg| \\
		&\quad + \frac{1}{2}\sum_{u_j,\dots,u_k\geq 0} \bigg| p_{\tilde{Z}_j}(u_{j-1},u_j) \prod_{i=j+1}^{k} p_{Z_i}(u_{i-1},u_i)
			- \prod_{i=j}^{k} p_{\tilde{Z}_i}(u_{i-1},u_i) \bigg| \\
		&\leq \frac{1}{2}\sum_{u_j = 0}^{\infty} \Big| p_{Z_j}(u_{j-1},u_j) - p_{\tilde{Z}_j}(u_{j-1},u_j) \Big| \\
		&\quad + \frac{1}{2} \sum_{u_j = 0}^{\infty} p_{\tilde{Z}_j}(u_{j-1},u_j) \sum_{u_{j+1},\dots,u_k\geq 0}
			\bigg| \prod_{i=j+1}^k p_{Z_i}(u_{i-1},u_i) - \prod_{i=j+1}^{k} p_{\tilde{Z}_i}(u_{i-1},u_i) \bigg| \\
		&= ||\L_{Z_j | Z_{j-1} = u_{j-1}} - \L_{\tilde{Z}_j | \tilde{Z}_{j-1} = u_{j-1}}||_{TV} \\
		&\quad + \sum_{u_j = 0}^{\infty} p_{\tilde{Z}_j}(u_{j-1},u_j)
			\cdot ||\L_{(Z_{j+1}, \dots, Z_k) | Z_j = u_j} - \L_{(\tilde{Z}_{j+1}, \dots, \tilde{Z}_k) | \tilde{Z}_j = u_j}||_{TV}.
	\end{align*}
    \endgroup
	An application of \thref{OneStepTVDBound} produces a bound for the first of the two terms above, showing that
	$||\L_{Z_j | Z_{j-1} = u_{j-1}} - \L_{\tilde{Z}_j | \tilde{Z}_{j-1} = u_{j-1}}||_{TV} \leq \J(u_{j-1})$.
	The second term is dealt with as follows: for some $N\in\N_1$ such that $N < \phi(u_{j-1}) \cdot \tilde{m}$,
	we split the series into the partial sum up to $N$, and the series beginning from $N+1$.
	Then, noting that the total variation distance cannot exceed a value of one,
	and is also assumed to be bounded by the decreasing function $\K_{j+1}$, we have
	\begin{align*}
		& ||\L_{(Z_j, \dots, Z_k) | Z_{j-1} = u_{j-1}} - \L_{(\tilde{Z}_j, \dots, \tilde{Z}_k) | \tilde{Z}_{j-1} = u_{j-1}}||_{TV} \\
		&\leq \J(u_{j-1}) + \sum_{u_j = 1}^N p_{\tilde{Z}_j}(u_{j-1},u_j) \cdot 1
			+ \sum_{u_j = N+1}^{\infty} p_{\tilde{Z}_j}(u_{j-1},u_j) \cdot \K_{j+1}(u_j) \\
		&\leq  \J(u_{j-1}) + \sum_{u_j = 1}^N p_{\tilde{Z}_j}(u_{j-1},u_j)
			+ \K_{j+1}(N+1) \cdot \sum_{u_j = N+1}^{\infty} p_{\tilde{Z}_j}(u_{j-1},u_j)\\
		&\leq \J(u_{j-1}) + \P(\tilde{Z}_j \leq N | \tilde{Z}_{j-1} = u_{j-1}) + \K_{j+1}(N+1). \\
	\end{align*}
	Applying Chebyshev's inequality yields
	\begin{align*}
		\P(\tilde{Z}_j \leq N | \tilde{Z}_{j-1} = u_{j-1})
			&\leq \P\big( | \tilde{Z}_j - \phi(u_{j-1}) \cdot \tilde{m} |
				\geq \phi(u_{j-1}) \cdot \tilde{m} - N \,\big|\, \tilde{Z}_{j-1} = u_{j-1} \big) \\
		&\leq \frac{\phi(u_{j-1}) \cdot \tilde{\sigma}^2}{\big(\phi(u_{j-1}) \cdot \tilde{m} - N\big)^2},
	\end{align*}
	so, by setting $N := \lfloor \alpha \cdot \phi(u_{j-1}) \cdot \tilde{m} \rfloor$, we arrive at the bound
	\begin{align*}
		& ||\L_{(Z_j, \dots, Z_k) | Z_{j-1} = u_{j-1}} - \L_{(\tilde{Z}_j, \dots, \tilde{Z}_k) | \tilde{Z}_{j-1} = u_{j-1}}||_{TV} \\
		&\leq \J(u_{j-1}) + \frac{\phi(u_{j-1}) \cdot \tilde{\sigma}^2}{\big( \phi(u_{j-1}) \cdot \tilde{m} - N \big)^2}
			+ \K_{j+1}(N+1) \\
		&= \J(u_{j-1}) + \frac{\phi(u_{j-1}) \cdot \tilde{\sigma}^2}
			{\big( \phi(u_{j-1}) \cdot \tilde{m} - \lfloor \alpha \cdot \phi(u_{j-1}) \cdot \tilde{m} \rfloor \big)^2}
			+ \K_{j+1}\big( \lfloor \alpha \cdot \phi(u_{j-1}) \cdot \tilde{m} \rfloor + 1 \big) \\
		&\leq \J(u_{j-1}) + \frac{\phi(u_{j-1}) \cdot \tilde{\sigma}^2}
			{\big( \phi(u_{j-1}) \cdot \tilde{m} - \alpha \cdot \phi(u_{j-1}) \cdot \tilde{m} \big)^2}
			+ \K_{j+1}\big( \alpha \cdot \phi(u_{j-1}) \cdot \tilde{m} \big) \\
		&= \J(u_{j-1}) + \frac{\tilde{\sigma}^2 \cdot \phi(u_{j-1})}{(1 - \alpha)^2 \tilde{m}^2 \cdot \phi(u_{j-1})^2}
			+ \K_{j+1}\big( \alpha \cdot \phi(u_{j-1}) \cdot \tilde{m} \big) \\
		&\leq \J(u_{j-1}) + \frac{\tilde{\sigma}^2}{(1-\alpha)^2\tilde{m}^2 h \cdot u_{j-1}}
			+ \K_{j+1}\big( \alpha\tilde{m}h \cdot u_{j-1} \big) \\
		&= \K_j(u_{j-1}).
	\end{align*}
	In addition, since we know from \thref{OneStepTVDBound} that $\J$ is decreasing in $u_{j-1}$,
	since it is clear that $\tilde{\sigma}^2 ((1-\alpha)^2\tilde{m}^2 h \cdot u_{j-1})^{-1}$ is decreasing in $u_{j-1}$,
	and since we have by assumption that $\K_{j+1}$ is decreasing in $u_{j-1}$ (and therefore in $\alpha\tilde{m}h \cdot u_{j-1}$),
	then so too is $\K_j$ decreasing in $u_{j-1}$.
    \qed
\end{proof}
\newline

\begin{proof}[\textbf{Proof of \thref{ClosedFormKStepTVDBound}}.]
	Assume that the conditions of \thref{KStepTVDBound} hold. Then, by defining
	\[
		b := \frac{\sqrt{2}\left( 3R + (1 + h)\tilde{\sigma}^2 \right)}{\tilde{\sigma}^2 (h\wedge 1)\sqrt{\pi\eta}}
		+ \frac{\left( 5\sqrt{2\pi} + \frac{3\pi}{2} \right) (1 + h) R + h\tilde{\sigma}^2}{\tilde{\sigma}^3\sqrt{2\pi h^3}},
	\]
	it follows immediately from \thref{OneStepTVDBound}
	that $||\L_{Z_1|Z_0 = z} - \L_{\tilde{Z}_1|\tilde{Z}_0 = z}||_{TV} \leq \frac{b}{\sqrt{z}}$ when $k = 1$. \\
	
	For $k \geq 2$, both terms in the bound of \thref{KStepTVDBound} are geometric sums as long as $\alpha \tilde{m} h \neq 1$
	(and of couse we can always choose $\alpha$ to ensure that $\alpha \tilde{m} h \neq 1$), so that we can write
	\begin{align*}
		&||\L_{(Z_1, \dots, Z_k)|Z_0 = z} - \L_{(\tilde{Z}_1, \dots, \tilde{Z}_k)|\tilde{Z}_0 = z}||_{TV} \\
		&\leq \sum_{i=0}^{k-1} \J\big( ( \alpha \tilde{m} h )^i \cdot z \big)
			+ \frac{\tilde{\sigma}^2}{(1 - \alpha)^2 \tilde{m}^2 h \cdot z} \cdot \sum_{i=0}^{k-2} ( \alpha \tilde{m} h )^{-i} \\
		&\leq \frac{b}{\sqrt{z}} \cdot \sum_{i=0}^{k-1} ( \alpha \tilde{m} h )^{-\frac{i}{2}}
			+ \frac{\tilde{\sigma}^2}{(1 - \alpha)^2 \tilde{m}^2 h \cdot z} \cdot \sum_{i=0}^{k-2} ( \alpha \tilde{m} h )^{-i} \\
		&= \frac{b}{\sqrt{z}} \cdot \frac{1 - (\alpha\tilde{m}h)^{-\frac{k}{2}}}{1 - (\alpha\tilde{m}h)^{-\frac{1}{2}}}
			+ \frac{\tilde{\sigma}^2}{(1 - \alpha)^2 \tilde{m}^2 h \cdot z}
			\cdot \frac{1 - (\alpha\tilde{m}h)^{-k+1}}{1 - (\alpha\tilde{m}h)^{-1}} \\
		&= \frac{b \sqrt{\alpha\tilde{m}h} \big( 1 - (\alpha\tilde{m}h)^{-\frac{k}{2}} \big)}{(\sqrt{\alpha\tilde{m}h} - 1) \cdot \sqrt{z}}
			+ \frac{\alpha\tilde{\sigma}^2 \big( 1 - (\alpha\tilde{m}h)^{-k+1} \big)}{(1 - \alpha)^2(\alpha\tilde{m}h - 1)\tilde{m} \cdot z} \\
		&= \frac{c_1 \big| 1 - (\alpha\tilde{m}h)^{-\frac{k}{2}} \big|}{\sqrt{z}}
			+ \frac{c_2 \big| 1 - (\alpha\tilde{m}h)^{-k+1} \big|}{z},
	\end{align*}
	where we set $c_1 := \Big| \frac{b \sqrt{\alpha\tilde{m}h}}{\sqrt{\alpha\tilde{m}h} - 1} \Big|$
	and $c_2 := \Big| \frac{\alpha\tilde{\sigma}^2}{(1 - \alpha)^2(\alpha\tilde{m}h - 1)\tilde{m}} \Big|$.
    \qed
\end{proof}
\newline

\begin{proof}[\textbf{Proof of \thref{WhenPSDBPMustMatchDCBP}.}]
	Given a DCBP, \thref{MatchingMomentsEquations} tells us that any matching PSDBP must have an offspring distribution with mean
	$m(z) = \frac{\tilde{m} \cdot \phi(z)}{z}$ and variance $\sigma^2(z) = \frac{\tilde{\sigma}^2 \cdot \phi(z)}{z}$, $z\in\N_1$.
	The result then follows directly from \thref{MinimumDiscreteVariance}, below.
    \qed
\end{proof}

\begin{lemma}\label{MinimumDiscreteVariance}
    Suppose $\alpha,\beta > 0$ and $d = \alpha - \lfloor\alpha\rfloor$.
	Then there exists a random variable $X$ on $\N_0$ with $\E X = \alpha$ and $\V(X) = \beta$
	if and only if $\beta \geq d(1-d)$.
\end{lemma}

\begin{proof}
	Fix $\alpha>0$. We need to show that
	\begin{enumerate}
		\item[(i)] there is no distribution on $\N_0$ with mean $\alpha$ and variance $\beta < d(1-d)$, and,
		\item[(ii)] for all $\beta \geq d(1-d)$ there exists an $X$ such that $\E X = \alpha$ and $\V(X) = \beta$.
	\end{enumerate}
    We first prove (i).
	Let $Y_{\{y_1, y_2\}}$ be a random variable with $\E Y_{\{y_1, y_2\}} = \alpha$ and support on $\{y_1, y_2\}$, where $y_1, y_2\in\N_0$ and $y_1 \leq \alpha < y_2$.
    Letting $p_{1} := \P( Y_{\{y_1, y_2\}} = y_1 )$ and $p_{2} := \P( Y_{\{y_1, y_2\}} = y_2 )$,
	we have the simultaneous equations 
    \[
        p_{1} + p_{2} = 1 \quad \text{ and } \quad y_1 p_{1} + y_2 p_{2} = \alpha,
    \]
    which imply that
	\begin{equation*}
		p_{1} = \frac{y_2 - \alpha}{y_2 - y_1} \quad\text{and}\quad p_{2} = \frac{\alpha - y_1}{y_2 - y_1}. 
	\end{equation*}	
	The variance of $Y_{y_1,y_2}$ is then given by 
	\begin{align*}
		\V(Y_{\{y_1, y_2\}})
		&=  y_1^2 \cdot \frac{y_2 - \alpha}{y_2 - y_1} + y_2^2 \cdot \frac{\alpha - y_1}{y_2 - y_1} - \alpha^2 \nonumber \\
		&= -y_1y_2 + \alpha(y_1 + y_2) - \alpha^2.
	\end{align*}
	When $y_2 > \alpha$, $\V(Y_{\{y_1, y_2\}})$ is minimised w.r.t.\ $y_1$ by taking $y_1$ as large as possible, and when $y_1 \leq \alpha$, $\V(Y_{\{y_1, y_2\}})$ is minimised w.r.t.\ $y_2$ by taking $y_2$ as small as possible.
    Since $y_1, y_2\in\N_0$ and $y_1 \leq \alpha < y_2$, this implies that $\V(Y_{\{y_1, y_2\}})$ is minimised w.r.t. $y_1$ and $y_2$ when $y_1 = \lfloor \alpha \rfloor$
	and $y_2 = \lfloor \alpha \rfloor + 1$. Consequently, 
	\begin{align*}
		\V\big(Y_{\{y_1,y_2\}}\big) &\geq \V\big(Y_{\{\lfloor\alpha\rfloor, \lfloor\alpha\rfloor + 1\}}\big) \\
		&= -\lfloor\alpha\rfloor(\lfloor\alpha\rfloor + 1) + \alpha(\lfloor\alpha\rfloor + (\lfloor\alpha\rfloor + 1)) - \alpha^2 \\
		&= (\alpha - \lfloor\alpha\rfloor)(1 - (\alpha - \lfloor\alpha\rfloor)) \\
		&= d(1 - d).
	\end{align*}
    Because the family of distributions $\big\{ Y_{\{y_1, y_2\}} \big\}_{y_1, y_2\in\N_0,\;y_1 \leq \alpha < y_2}$ contains every candidate for $X$ with support on two natural values (referred to as a two-point distribution),
    we conclude that if $\beta < d(1-d)$, then $X$ cannot possess a two-point distribution. \\

    Next consider the $k$-point distributions, for $k > 2$.
    Every $k$-point distribution, $K_{\{n_1, \dots, n_k\}}$, with $n_1, \dots, n_k \in \N_0$ and $\E K_{\{n_1, \dots, n_k\}} = \alpha$
    can be expressed as a mixture of $k-1$ two-point distributions, all with mean $\alpha$.
    This follows from the linearity of expectation, and can be formalised by modifying the arguments in the proof of \cite[Theorem 1]{kronmal79}.

    Specifically, we can write
    \[
        K_{\{n_1, \dots, n_k \}} \stackrel{d}{=} \sum_{i=1}^{k-1} Y_{\{ u_i, v_i \}}\1\{W=i\},
    \]
    where $k \in \mathbb{N}_0$, $W$ is a random variable taking values in $1, \dots, k-1$,
    $u_i, v_i \in \{n_1, \dots, n_k \}$, $\mathbb{E}Y_{\{u_i, v_i\}} = \alpha$ for all $i=[k-1]$,
    and $W$ and each $Y_{\{u_i, v_i\}}$ are all mutually independent.

    We then obtain
	\begin{align*}
		\V(K_{\{n_1, \dots, n_k \}})
        &= \sum_{i=1}^{k-1} \P(W=i) \cdot \E Y_{\{ u_i, v_i \}}^2 - \alpha^2 \\
		&= \sum_{i=1}^{k-1} \P(W=i) \cdot \V(Y_{\{ u_i, v_i \}}) \\
		&\geq \sum_{i=1}^{k-1} \P(W=i) \cdot d(1-d) \\
		&= d(1-d),
	\end{align*}
    and can therefore conclude that if $\beta < d(1-d)$, then $X$ cannot possess a $k$-point distribution. \\

	This same approach can be extended to discrete distributions with infinite support and finite mean $\alpha$.
    Let $I_{\{n_1, n_2, \dots\}}$ have such a distribution, where $0 \leq n_1 < n_2 < \dots$ and $n_i \in \N_0$ for all $i \in \N_1$.
    For sufficiently large $N \in \N_1$, we can repeat the same trick of decomposing $I_{\{n_1, n_2, \dots\}}$ into two-point distributions,
    but consider the tail $\{n_N, n_{N+1}, \dots\}$ as a single `point' when performing the decomposition.

    Specifically, we can write
    \[
        I_{\{n_1, n_2, \dots\}} \stackrel{d}{=} \sum_{i=1}^{N-2} Y_{\{u_i,v_i\}} \1\{ W = i \} 
            + I_{\{n^*, n_N, n_{N+1}, \dots\}} \1\{W = N - 1\},
    \]
    where $I_{\{n^*, n_N, n_{N+1}, \dots\}}$ is a random variable supported on the tail $\{n_N, n_{N+1}, \dots\}$ and a value $n^* \in \{n_1, \dots, n_{N-1} \}$ such that $(n^* - \alpha) \cdot p_{n^*} = -\sum_{i=N}^{\infty} (n_i - \alpha) \cdot p_{n_i}$
    (using the notation $p_n := \P( I_{\{n_1, n_2, \dots\}} = n )$).
    Analogous to the $k$-point case, we also assume that $u_i, v_i \in \{n_1, \dots, n_{N-1} \}$, $u_i < v_i$, for all $i \in [N-2]$,
    $W$ is a mixing distribution taking values in $1, \dots, N-1$, and all the random variables are mutually independent.

    Then, in the same manner as for $\V(K_{\{n_1, \dots, n_k \}})$, we have
    \[
        \V\big( I_{\{n_1, n_2, \dots\}} \big) \geq \sum_{i=1}^{N-2} \P(W = i) \cdot d(1-d).
    \]
    Since $\sum_{i=N}^{\infty} (n_i - \alpha) \cdot p_{n_i} \to 0$ as $N \to \infty$ and $\P(W = N-1) \to 0$ as $N \to \infty$,
    we can take the limit as $N \to \infty$ to get
    $\V\big( I_{\{n_1, n_2, \dots\}} \big) \geq d(1-d)$.
    
    With all possible cases considered, we see that there is no distribution on $\N_0$ with $\beta < d(1-d)$. \\

	It remains to prove (ii). Fix $\alpha,\beta\in\R_{>0}$, with $\beta \geq d(1-d)$.
	We want to show that there exists a RV $X$ on $\N_0$ such that $\E X = \alpha$ and $\V(X) = \beta$.
	If $\beta = d(1-d)$ we can simply take $X = Y_{\{\lfloor\alpha\rfloor, \lfloor\alpha\rfloor + 1\}}$.
	
	Otherwise, if $\beta > d(1-d)$, we define
	\begin{equation*}
		Y^* = \begin{cases}
			Y_{\{ \lfloor\alpha\rfloor, y \}},\; y:=
				\big\lceil \frac{\beta + \alpha^2 - \alpha\lfloor\alpha\rfloor}{\alpha - \lfloor\alpha\rfloor} \big\rceil & \text{if } \alpha\not\in\N_1 \\
			Y_{\{ \alpha - 1, \lceil \alpha + \beta \rceil \}} & \text{otherwise}
		\end{cases}
	\end{equation*}
	and use our previous equation for the variance of a two-point distribution to see that $\V(Y^*) \geq \beta$,
	while by construction $\E Y^* = \alpha$.
	
	Therefore, defining $W\sim\text{Ber}(q)$, $q\in[0,1]$, to be a Bernoulli RV with arbitrary parameter,
	the mixture
    \[
        M_q := Y_{\{\lfloor\alpha\rfloor, \lfloor\alpha\rfloor + 1\}}\cdot \1\{W=0\} + Y^*\cdot \1\{W=1\}
    \]
	(for $Y_{\{\lfloor\alpha\rfloor, \lfloor\alpha\rfloor + 1\}}$, $Y^*$, and $W$ mutually independent) has mean $\alpha$ and variance
	\begin{equation*}
		\V(M_q) = (1-q) \cdot d(1-d) +  q \cdot \V(Y^*).
	\end{equation*}
	By varying $q$ over $[0, 1]$ we can continuously scale $\V(M_q)$ over the range $[d(1-d),\V(Y^*)]$,
	so there exists a $q^*\in[0,1]$ such that $\V(M_{q^*}) = \beta$.
	We can then simply take $X \equiv M_{q^*}$, and so $\E X = \alpha$ and $\V(X) = \beta$.
    \qed
\end{proof}
\newline

\begin{proof}[\textbf{Proof of \thref{WhenDCBPMustMatchPSDBP}.}]
    We will show that satisfying the conditions of \thref{WhenDCBPMustMatchPSDBP} is equivalent to satisfying Equations
    \eqref{eqn:DCBPPSDBPMatchingMean} and \eqref{eqn:DCBPPSDBPMatchingVariance} of \thref{MatchingMomentsEquations},
    i.e.\ $z \cdot m(z) = \tilde{m} \cdot \phi(z)$, and $z \cdot \sigma^2(z) = \st^2 \cdot \phi(z)$. There are four cases to consider:
    \begin{enumerate}
        \item[(a)] When $z = 0$, taking $\phi(0) = 0$ satisfies \eqref{eqn:DCBPPSDBPMatchingMean} and \eqref{eqn:DCBPPSDBPMatchingVariance}.
        \item[(b)] For $z \in \N_1$, $m(z) = 0 \implies \sigma^2(z) = 0$,
            and taking $\phi(z) = 0$ for any such $z$ satisfies \eqref{eqn:DCBPPSDBPMatchingMean} and \eqref{eqn:DCBPPSDBPMatchingVariance}.
        \item[(c)] For $z \in \N_1$, $\sigma^2(z) = 0$ and $m(z) \neq 0 \implies m(z) \in 
            \N_1$ (by \thref{MinimumDiscreteVariance}).
            For such a $z$, since we require $\phi(z) \neq 0$ to satisfy \eqref{eqn:DCBPPSDBPMatchingMean},
            we must have $\st^2 = 0$ and $\mt \in \N_1$ (using \thref{MinimumDiscreteVariance}).
            Hence, if we take $\mt = 1$ and $\phi(z) = z \cdot m(z)$, 
            \eqref{eqn:DCBPPSDBPMatchingMean} and \eqref{eqn:DCBPPSDBPMatchingVariance} are satisfied.
        \item[(d)] Otherwise, for $z \in \N_1$, $m(z) \neq 0$, $\sigma^2(z) \neq 0$, we     require $\mt \neq 0$, $\st \neq 0$, and $\phi(z) \neq 0$,
            and can equate $m(z) / \mt = \phi(z) / z$ and $\sigma^2(z) / \st^2 = \phi(z) / z$ to get $m(z) / \mt = \sigma^2(z) / \st^2$,
            yielding $m(z) = k \cdot \sigma^2(z)$ when we set $k := \mt / \st^2$.
    \end{enumerate}
    While cases (a) and (b) impose no restrictions on the values of $\mt$ and $\st$, (c) and (d) impose mutually contradictory restrictions.

    Hence, if, for any $z \in \N_1$, $\sigma^2(z) = 0$ and $m(z) \neq 0$, we need $\sigma^2(z) = 0$ for all attainable $z \in \N_1$,
    which is sufficient to satisfy \thref{WhenDCBPMustMatchPSDBP}.

    Otherwise we require that there exists a $k > 0$ such that, for all attainable $z \in \N_1$, $m(z) = k \cdot \sigma^2(z)$:
    Condition (i) of \thref{WhenDCBPMustMatchPSDBP}.
    This, however, is only a necessary condition. We need to strengthen it further to create a necessary and sufficient condition for satisfying \thref{WhenDCBPMustMatchPSDBP}. \\

    Suppose that a $k > 0$ satisfying Condition (i) exists. It turns out that the need to satisfy \eqref{eqn:DCBPPSDBPMatchingMean} implies Condition (ii).
    
    To show this, suppose that there exists a constant $h > 0$ such that for each attainable $z \in \N_1$, there is an $n(z) \in \N_0$ with $m(z) = h \cdot n(z)$.
    That is, suppose that the set $H := \big\{ h \in \R_{>0} : m(z) \in h \cdot \N_0 \text{ for all attainable } z \in \N_1 \big\}$ is non-empty.
    Then we can satisfy \eqref{eqn:DCBPPSDBPMatchingMean} by taking $\phi(z) := z \cdot n(z)$ and $\mt := h$.

    Suppose, on the other hand, that no such $h$ exists, so that $H$ is empty.
    Then there exist attainable $z_1, z_2 \in \N_1$ such that the sets
    $H_1 := \big\{ h_1 \in \R_{>0} : m(z_1) \in h_1 \cdot \N_0 \big\}$ and $H_2 := \big\{ h_2 \in \R_{>0} : m(z_2) \in h_2 \cdot \N_0 \big\}$ are disjoint
    (which requires that $m(z_1) \neq 0$ and $m(z_2) \neq 0$).

    For $z = z_1$, we can rewrite \eqref{eqn:DCBPPSDBPMatchingMean} as $\mt = \frac{z_1 \cdot m(z_1)}{\phi(z_1)}$. Since $z_1 \in \N_1$ and we require $\phi(z_1) \in \N_1$,
    we have that $\mt \in \left\{ \frac{z_1 \cdot m(z_1)}{x} : x \in \N_1 \right\}$.
    Rewriting $m(z_1) = h_1^{(x)} \cdot z_2 \cdot x$, where $h_1^{(x)} \in H_1$,
    we find that $\mt \in \left\{ h_1^{(x)} \cdot z_1 \cdot z_2 : x \in \N_1 \right\} \subseteq z_1 \cdot z_2 \cdot H_1$.

    We can repeat the above for $z = z_2$ to find that $\mt \subseteq z_1 \cdot z_2 \cdot H_2$.
    But this is a contradiction, since $(z_1 \cdot z_2 \cdot H_1) \cap (z_1 \cdot z_2 \cdot H_2) = z_1 \cdot z_2 \cdot (H_1 \cap H_2) = \varnothing$. \\

    It is worth emphasising some consequences of Condition (ii):
    ignoring the trivial case where $m(z) = 0$ for all attainable $z$, if $H$ is non-empty then it contains a maximal element;
    since $h \leq m(z)$ (i.e.\ taking $m(z) = h \cdot 1$) for all $z \in \N_1$ such that $m(z) \neq 0$, $h^{max} \leq m(z^{min})$,
    where we define $m(z^{min}) := \min\{ m(z) : z\in\N_1 \text{ s.t. } m(z) \neq 0 \}$
    and $h^{max} \in H$ represents the maximal element.
    
    If $H$ is non-empty, then it will also contain countably many elements;
    $h \in H \implies \frac{h}{x} \in H$ for any $x \in \N_1$. \\

    For a DCBP satisfying Conditions (i) and (ii), (iii) will follow from \thref{MinimumDiscreteVariance}.
    For such a DCBP, Condition (i) allows us to simplify \eqref{eqn:DCBPPSDBPMatchingVariance} to
    \[
        z \cdot m(z) = k \cdot \st^2 \cdot \phi(z),
    \]
    where, from (ii), we take $m(z) = h \cdot n(z)$ and $\phi(z) = z \cdot n(z)$.
    Hence we require $\st^2 = \frac{h}{k}$, for some $h \in H$.
    Then, by \thref{MinimumDiscreteVariance}, the DCBP can have matching moments if and only if
    \[
        \frac{h}{k} \geq \left( h - \lfloor h \rfloor \right)\left( 1 - h + \lfloor h \rfloor \right)
    \]
    for at least one $h \in H$.
	
	For values of $h \in H \cap [0, 1]$, $\left( h - \lfloor h \rfloor \right)\left( 1 - h + \lfloor h \rfloor \right) = h (1 - h)$,
	so the previous inequality becomes $\frac{1}{k} \geq (1 - h)$.
	If this holds for any $h \in H \cap [0, 1]$, it will hold for $\sup_{h<1}\{h\in H\}$. This implies Condition (iii).
    \qed
\end{proof}
\newline

\begin{proof}[\textbf{Proof of \thref{TVDEstimator}.}]
    Given $X$ and $Y$ are defined on a countable space, it follows from \eqref{eqn:TVDEquivalentSum} that
    \begin{align*}
        || \L_X - \L_Y ||_{TV} &= \frac{1}{2} \sum_{n\in\mathcal{X}} | \P(X = n) - \P(Y = n) | \\
            &= \frac{1}{2} \sum_{n\in\mathcal{X}} \frac{\big| \P(X = n) - \P(Y = n) \big|}{\P(X = n)} \cdot \P(X = n) \\
        &= \frac{1}{2} \cdot \E_X \left( \frac{| \L_X(X) - \L_Y(X) |}{\L_X(X)} \right).
    \end{align*}

    With $\theta := \frac{1}{2} \cdot \E_X \left( \frac{| \L_X(X) - \L_Y(X) |}{\L_X(X)} \right)$, it follows directly that $\E_X ( \hat{\theta}_N ) = \theta$.
    Hence $\hat{\theta}_N$ is an unbiased estimator for $\theta$, and, by the law of large numbers, for any $\epsilon > 0$,
    $\lim_{N \to \infty} \P\left( \big| \hat{\theta}_N - \theta \big| > \epsilon \right) = 0$,
    so it is a consistent one as well.
    \qed
\end{proof}

\section*{Acknowledgements}

Sophie Hautphenne would like to thank the Australian Research Council (ARC) for support through her Discovery Project DP200101281.

\bibliographystyle{plain}
\bibliography{main}

\end{document}